\documentclass[14pt]{amsart}
\usepackage{amsmath}
\usepackage{amsthm}
\usepackage{amssymb}
\usepackage{amscd}
\usepackage{amsfonts}
\usepackage{amsbsy}
\usepackage{epsfig,afterpage}
\usepackage[dvips]{psfrag}
\usepackage{subfigure}
\usepackage{epsfig}
\usepackage{amsmath}
\usepackage{srcltx}
\usepackage{amsfonts}
\usepackage{amssymb}
\usepackage{psfrag}
\usepackage{verbatim}
\usepackage{graphicx}
\usepackage{amsmath}
\usepackage{amsthm}
\usepackage{amssymb}
\usepackage{amscd}
\usepackage{amsfonts}
\usepackage{amsbsy}
\usepackage{epsfig}
\usepackage[dvips]{psfrag}
\usepackage{multirow}
\usepackage{graphics}
\usepackage{epstopdf}
\usepackage{overpic}
\usepackage{float}
\usepackage{multicol}

\usepackage{graphicx}
\usepackage{mathabx}

\newtheorem {theorem} {Theorem}%[section]
\newtheorem {proposition} [theorem]{Proposition}

\newtheorem {lemma}  [theorem]{Lemma}

\newtheorem {remark} [theorem]{Remark}

\usepackage{float}
\usepackage{tikz, wrapfig}
\tikzset{node distance=3cm, auto}
\allowdisplaybreaks

\begin{document}

\title[Poincaré Map for Non-smooth Welander's Ocean Convection Model]
{Study of Poincar\'{e} Map and limit cycles for Non-smooth Welander's Ocean Convection Model}

\author[Y. R. Carvalho] {Yagor Romano Carvalho$^{1}$}
\address{$^1$ University of São Paulo, São Carlos, Brazil, and Unviersity of Texas at Dallas, Richardson, United States of America}
\email{$^1$yagor.carvalho@usp.br}

\author[L. F. S Golveia] {Luiz F.S. Gouveia$^2$}
\address{$^2$ São Paulo State University, São José do Rio Preto, Brazil, and Unviersity of Texas at Dallas, Richardson, United States of America }
\email{$^2$fernando.gouveia@unesp.br}

\author[R. McGehee] {Richard McGehee$^3$}
\address{$^{3}$ University of Minnesota, Minneapolis, United States of America }
\email{$^3$mcgehee@umn.edu}

%\subjclass[2010]{32A10, 34C20, 34A34, 34A36, 34C05.}

\keywords {piecewise systems, limit cycles, Welander model, parabolic points, Poincaré Map}
\date{}
\dedicatory{}
\maketitle

\vspace{-0.5cm}

\begin{abstract}
In this work, our primary goal is to study the Poincaré map and the existence of limit cycles for Welander's model that describes ocean convection. Welander developed two versions of his model, one with a smooth transition between convective states, and one with an abrupt non-smooth change. Our focus in this paper is to study the non-smooth model. Approaching through the Poincar\'{e} Map, we demonstrate analytically the bifurcation of a stable limit cycle surrounding an escaping segment. 

%Furthermore, we also apply our results to a smooth e example which is the point limit of the smooth example computed by Welander.
\end{abstract}

\section{Introduction}

Ocean circulation plays a very important role in Earth's climate system. For instance, heat and moisture are distributed around the globe through ocean circulation. Even small changes, in its patterns, can have considerable effects in both regional and global climates. Because of this, understanding this phenomenon is crucial in predicting events that may become catastrophic. In this sense, we can cite the modern phenomenon ``El Niño'' and the paleo event ``Dansgaard Oeschger''. These events can be linked to changing ocean dynamics, causing global changes in temperature, see \cite{CT,Ocean,Dyn}.   For example, after the Last Glacial Maximum (LGM), also referred the Last Glacial Coldest Period,  it has been observed Dansgaard Oeschger events, which are periodic weakening and strengthening of the Atlantic Meridional Overturning Circulation (AMOC). It is believed that those oscillations are caused by a sudden influx of freshwater into the Atlantic Ocean, and such a phenomenon can be well modeled as a relaxation oscillation in a low-dimensional model of temperature and salinity in the AMOC. On the other hand, when paying attention to the Gulf Stream,  in a shorter time scale, there is a periodic variability of smaller magnitude in its strength, see \cite{VAR,inter,using}. Since the Gulf Stream is a major transporter of heat to Europe, changes in this ocean circulation may imply consequences for systems relying on regular climate patterns, such as agriculture.  Therefore, the importance of the role played by ocean circulation in the dynamics of the global climate is undeniable. Then to model and explain any of its variations is obviously extremely important to the world, and for example,  the book by Kaper and Engle  \cite{KapEng2013} gives a good introduction to mathematical approaches in climate modeling.

In 1982, Welander \cite{welander} introduced a simple box model illustrating that interactions in salinity and temperature between the shallow and deep ocean can produce self-oscillations even with no external forcing, justifying its importance in the climate community. We are particularly interested in Welander's model since it is an example of a Fillipov system \cite{Fil1988} wherein the discontinuity set is created by a varying density causing the ocean to oscillate between stratification and mixing. Such a model is an ocean circulation box model, splitting the ocean into boxes, see Figure~\ref{boxmodel}, and this kind of model has been used to study small and large-scale changes in AMOC. Assuming that the water in each box is well mixed, ordinary differential equations describe the flow of water from one box to the other, and the only relevant variables are the temperature and salinity of the water in each box. In addition, the density is described as a linear combination of the variables. 
\begin{center}
	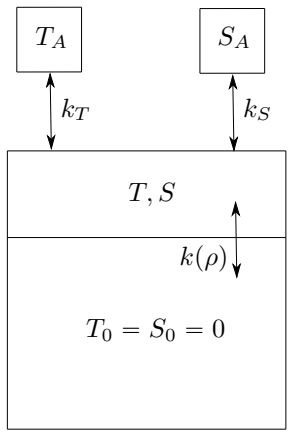
\begin{figure}[H]
		\begin{overpic}[scale=0.45]
			{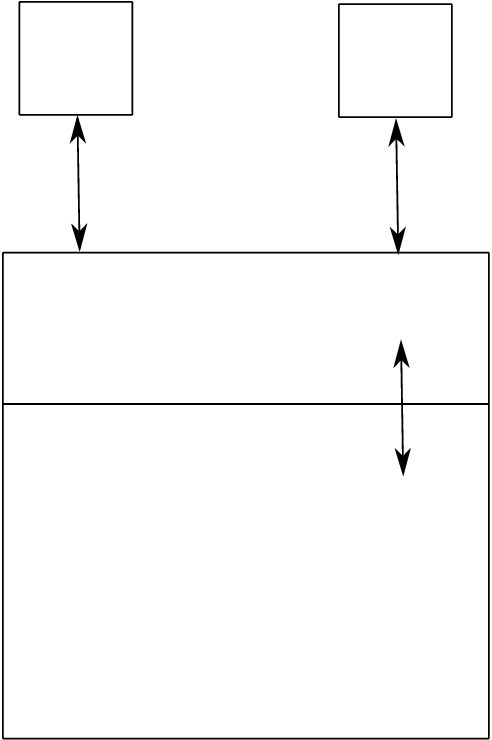} 
		   \put(7,91){$T_A$}
		   \put(50,91){$S_A$}
		   \put(13,74){$k_{T}$}
		   \put(56,74){$k_{S}$}
		   \put(29,54){$T,S$}
		   \put(41,39){$k(\rho)$}
		   \put(19,22){$T_{0}=S_{0}=0$}
		   \end{overpic}
		\vspace{0.4cm}
		\caption{Welander's box model. The temperature atmospheric $T_{A}$  and salinity forcing in the form of precipitation and evaporation $S_{A}$ force the temperature and salinity of the surface ocean. The surface box mixes convectively with the deep ocean via the convection function $k(\rho)$.}\label{boxmodel}
	\end{figure}
\end{center}

\vspace{-1.2cm}

Observe that the model of Welander separates the ocean into two regions surface and deep ocean. The interaction between those regions is dictated by the relative density of the water,  controlled by the temperature and salinity through a linear equation of state $\rho=-\alpha_1 T+\gamma S$. We assume that the deep ocean have constant density $T_{0}=S_{0}=\rho_{0}=0$. The surface is considered well mixed, where the circulation given by the equations must represent a large-scale behavior such that the density does not have an important impact. Then the equations for the temperature and salinity are given by
\begin{equation}\label{wel}
		%begin{cases}
		\dfrac{dT}{dt}=k_{T}(T_{A}-T)-\overline{k}(\rho)T, \quad
		\dfrac{dS}{dt}=k_{S}(S_{A}-S)-\overline{k}(\rho)S, \quad
		\rho=-\alpha_1 T+\gamma S,
	    %\end{cases}
\end{equation} 
where $T_{A}$ and $S_{A}$ are generic atmospheric forcing caused by precipitation, evaporation, and solar forcing, which are considered constants. In addition, $\alpha_1, \gamma, k_T$ and $k_S$ are also assumed constant. 

The main feature of Welander's model is the convective mixing nonnegative function $\overline{k}(\rho)$, which could be smooth or piecewise smooth. As the ocean is in general highly stratified, without much mixing between the layers, Welander in his model assumes that convection is large when the density of the surface ocean is large, small when the density is small, and the transition between these states occurs abruptly. These aspects are feasible assumptions, given that the ocean is highly stratified, not having much mixing between its layers, and since the exchange between the layers is considered to be fully turbulent, the same $k$ function is used for temperature and salinity. Therefore, the switching behavior allows some periodicity to appear in the model. To obtain a more simple system, we make a change of coordinates and time rescale given by
\begin{equation*}\
T^{*}=T/T_A, \quad S^{*}=S/S_A, \quad \rho^{*}=\rho/(\gamma S_{A}), \quad t^{*}=k_{T}t,
\end{equation*}
such that, renaming the parameters, it is possible to obtain the the following planar differential system  
\begin{equation}\label{wel1}
	%\begin{cases}
	\dot{T}=1-T-k(\rho)T, \qquad
	\dot{S}=\beta(1-S)-k(\rho)S, \qquad
	\rho=-\alpha T+S.
	%	\end{cases}
	\end{equation} 
where ``$\; \dot{} \;$'' denotes the derivative in relation to the time, and the new parameters are $\beta=k_S/k_T$,  $\alpha=(T_A \alpha_1)/(S_A \gamma)$ and $k(\rho)=\overline{k}(\rho)/k_T.$

For some parameter values, Welander showed an oscillatory movement in its system, using two specific functions $k(\rho)$. First, he chooses a continuous function described by
\begin{equation}\label{ksuave}
	k(\rho)=
\frac{1}{\pi} \, \arctan \left( \frac{\rho-\varepsilon}{a} \right) +\dfrac{1}{2},
\end{equation}
and numerically he finds a stable periodic orbit. In this work, we are interested in studying Welander's model \cite[Sec. 4]{welander} for the other function used by him, when $k(\rho)$ is  piecewise smooth  given by 
\begin{equation}\label{funck}
	\begin{aligned}
	k(\rho)=		k_1, \text{ if } \rho > \varepsilon, \qquad 	k(\rho)=		k_0, \text{ if } \rho \le  \varepsilon,
%		k(\rho)=		\left\{\begin{array}{l}
%			k_1, \text{ if } \rho > \varepsilon ,\\[5pt]
%			k_0, \text{ if } \rho \le \varepsilon,
%		\end{array} \right.
	\end{aligned}
\end{equation}
where $\varepsilon$ consist of a small parameter, $k_0$ is a small constant or zero and, $k_1$ is a large constant. Using the function \eqref{funck}, Welander does not discuss the mathematics behind this model with rigor, but also numerically finds a periodic solution for specific values of parameters. However, the mathematical mechanism behind this oscillation is different in the non-smooth model. 

As we will see in the course of the work the oscillations in the non-smooth case are caused by the fact that, for some parameter values, stable equilibrium points do not exist in the phase portrait (called virtual equilibrium points). In this way, when the system approaches such a virtual equilibrium, it crosses the discontinuity line and reverses its direction to approach the other virtual equilibrium point. We will see that this behavior can lead to a stable routine.

Therefore, our main objective here is to study the occurrences or not of these movements in an analytical way, for different values of parameters. More specifically, we would like to study the existence of isolated periodic orbits, so-called limit cycles. To keep the paper to a reasonable length, we restrict our attention to the case in which $\beta>0$.  

For planar piecewise smooth differential systems we have two kinds of limit cycles:  {\it sliding limit cycles} and {\it crossing limit cycles}. The first contains some segment of the discontinuity set, and the second one only contains some points of the discontinuity set. For more details on these kinds of differential systems see the books \cite{BerBudCham2008, Fil1988, MakLam2012, SimJohn2010}. Then, here, we are going to study crossing limit cycles, only mentioned as limit cycles, surrounding some sliding segment of the discontinuity line, for the piecewise smooth Welander's model \eqref{wel1} considering \eqref{funck}.

Therefore, the piecewise smooth Welander's system can be written as a linear planar piecewise smooth system,
\begin{equation}\label{eqlin}
	\begin{aligned}
		\dot{\mathrm{x}}=	A^{L}\mathrm{x}+b^{L}, \text{ when } f(x,y)\ge 0, \qquad 	\dot{\mathrm{x}}=	A^{R}\mathrm{x}+b^{R}, \text{ when } f(x,y)\leq0,
%	\left\{\begin{array}{l}
%			A^{L}\mathrm{x}+b^{L}, \text{ when } f(x,y)\ge 0,\\[5pt]
%			A^{R}\mathrm{x}+b^{R}, \text{ when } f(x,y)\leq0,
%		\end{array} \right.
	\end{aligned}
\end{equation}
where the discontinuity region (or discontinuity set) $\Sigma$ is given by $f^{-1}(0)$, with $0$ being a regular value of $f$. 

The study of a maximum quota of limit cycles for systems of kind \eqref{eqlin} has been the object of study for many years. Lum and Chuca in \cite{Chua}, under the continuity hypothesis conjectured that system \eqref{eqlin} had at most one limit cycle. In 1998, Freire et al in \cite{Freire1998} give the first proof for such conjecture. In 2021, Carmona et al \cite{CSN}, using new techniques, provide a new more simple proof for this conjecture. Han and Zhang in \cite{HanZhang} gave the first example of a piecewise linear differential system with two limit cycles. And, in the year of 2012, Huan and Yang \cite{SongYang} using a numerical argument showed a system with three limit cycles. However, LLibre and Ponce \cite{LP} showed a first analytical proof of a linear piecewise smooth system with three limit cycles. Actually, there are many works that provided examples with three limit cycles for this kind of differential system, see \cite{BPT, MLNT,FPTF,FPTF2,NLT,NT}. 

Using the theory developed in \cite{node}, remember that we are assuming $\beta>0$, and defining the following new parameters 
\begin{equation*}
	\begin{aligned}
		&\alpha^L_{\beta,\varepsilon}=\dfrac{-(1+k_0)(\beta \, (\varepsilon-1)+k_0 \, \varepsilon)}{k_0+ \beta}, \quad \alpha^R_{\beta,\varepsilon}=\dfrac{-(1+k_1)(\beta \, (\varepsilon-1)+k_1 \, \varepsilon)}{k_1+ \beta},\\
	\end{aligned}
\end{equation*}
our main results are: 

\begin{theorem}\label{noperio}
	If $\alpha \, (1-\beta) = 0$,  or  $\alpha \, (1-\beta) \neq 0$ satisfying $\alpha \geq \alpha^L_{\beta,\varepsilon}$ or $\alpha \leq \alpha^R_{\beta,\varepsilon}$, then the piecewise smooth Welander's system does not have crossing periodic orbits.
\end{theorem}

\begin{theorem}\label{limit}
	If $\alpha \, (1-\beta) \neq 0$ such that $\alpha < \alpha^L_{\beta,\varepsilon}$ and  $\alpha > \alpha^R_{\beta,\varepsilon}$  then we have following statements for the piecewise smooth Welander's system:
	\begin{itemize}
		\item[(i) ]the system does not have crossing periodic orbits for $\varepsilon \geq 0$,
		
		\item[(ii)] the system has a unique and stable crossing periodic for $\varepsilon <0$.
	\end{itemize}
\end{theorem}	

This paper is organized as follows. In Section~\ref{pre}, we present preliminary concepts and results, and we show a normal form for Weelander's system. The Section~\ref{PMW} is dedicated to developing the Poincar\'{e} Map and studying its properties. Finally, in Section~\ref{LMW}, we prove Theorems~\ref{noperio} and \ref{limit}, and we exhibit a non-smooth example of the Welander model with a stable limit cycle, which is the point limit of the smooth example computed by Welander with the function \eqref{ksuave}.

\section{Preliminaries}\label{pre}

In this section, we review some definitions and make some accurate analyses of the piecewise smooth Welander's system that will be used to prove our main results. First, we will review some definitions about piecewise smooth systems. A piecewise smooth system is a system given by
\begin{equation}\label{eqpie}
	\begin{aligned}
	Z(x,y)=
		Z^{+}(x,y), \text{ when } f(x,y)\ge 0, \;
		Z(x,y)=Z^{-}(x,y), \text{ when } f(x,y)\le0,
%	Z(x,y)=	\left\{\begin{array}{l}
%			Z^{+}(x,y), \text{ when } f(x,y)\ge 0,\\[5pt]
%			Z^{-}(x,y), \text{ when } f(x,y)\le0,
%		\end{array} \right.
	\end{aligned}
\end{equation}
where $f: \mathbb{R}^2 \rightarrow \mathbb{R}$ is a $C^{1}$ function such that $0$ is a regular value. The discontinuity curve is given by $\Sigma=f^{-1}(0)$, and $Z^{\pm}=(X^{ \pm}, Y^{\pm})$. Then, the manifold $\Sigma$ divides the plane into two half-planes, $\Sigma^\pm$ and the trajectories on $\Sigma$ are defined as follows. 

Given a point $p \in \Sigma$, we say that $p$ is a crossing point if, and only if $Z^{+}f(p) \cdot Z^{-}f(p)>0$ where $Z^{\pm}f(p)= \langle \nabla f(p), Z^{\pm}(p)\rangle$ is the Lie Derivative. Moreover, we have $p$ a positive (negative) crossing when $Z^{-}f(p)>0$ and $Z^{+}f(p)>0$  ($Z^{-}f(p)<0$ and $Z^{+}f(p)<0$). And we say that $p$ is sliding (escaping) point if  $Z^{+}f(p) \cdot Z^{-}f(p)<0$ such that $Z^{+}f(p)<0$ and $Z^{-}f(p)>0$ ($Z^{+}f(p)>0$ and $Z^{-}f(p)<0$).  For crossing points, the solutions in relation to $\Sigma$, approach on one side and leave on the other, and so, its solutions are given by the concatenation of the solutions in $\Sigma^{\pm}$. On the other hand, for sliding or escaping points, the solutions approach or push away of both sides of $\Sigma$, then we need a solution which is continued for future or past in the discontinuity set, making sense. In this way, Fillipov proposed sliding dynamics for this case, which is governed by the Filippov sliding vector field, 
\begin{equation*}\label{eqcampodeslizante}
	Z^s(p)=(1-\lambda)Z^-(p)+\lambda Z^+(q), \quad \mbox{for} \; \lambda=\lambda(p)=Z^{-}f(p)/(Z^{-}f(p)-Z^{+}f(p))
\end{equation*}
and $\lambda$ is such that $Z^{s}f(p)=0$. Therefore, either sliding or escaping points follow the solutions of the Filippov sliding vector field in the appropriate direction. For more details see \cite{Fil1988} and, the configuration of the points $p\in \Sigma$, mentioned before, are illustrated in Figure \ref{ch4:fi:slidingsegment}.
\begin{figure}[h]
	\begin{center}
		\vspace{-0.3cm}
		\includegraphics[scale=0.225]{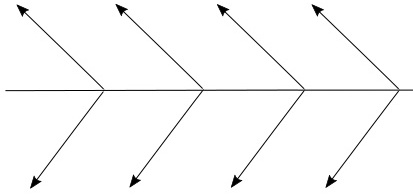}\quad \includegraphics[scale=0.225]{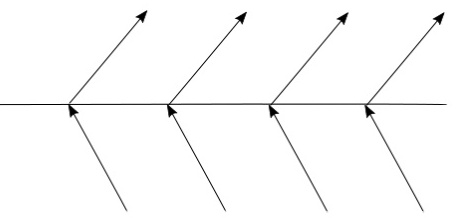}\quad \includegraphics[scale=0.225]{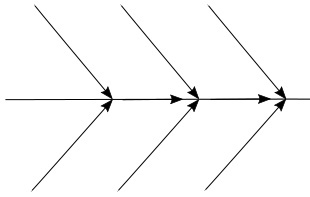}
	\end{center}
	\caption{Escaping, crossing and sliding segments}\label{ch4:fi:slidingsegment}
\end{figure}

\vspace{-0.2cm}

It is noteworthy that there are other different types of points in the discontinuity set $\Sigma$. For example, we say that $p \in \Sigma$ is a fold point of order $n$ in $Z^{+}$ ($Z^{-}$) with $n\in \mathbb{N}$, $n$ even, if $Z^{+}f(p)=0$ and $(Z^{+})^{n}f(p) \neq 0$ ($Z^{-}f(p)=0$ and $(Z^{-})^{n}f(p) \neq 0$), where $(Z^{\pm})^nf(p)= \langle  \nabla (Z^{\pm})^{n-1}(p), Z^{\pm}(p) \rangle$ for $n \ge 2$. The fold point for the vector field $Z^{+}$ ($Z^{-}$), is said visible or invisible, respectively, if 
$(Z^{+})^{n}f(p)>0 \, ((Z^{-})^{n}f(p)<0) \; \mbox{or} \; (Z^{+})^{n}f(p)<0 \, ((Z^{-})^{n}f(p)>0).$

Moreover, we say that $p \in \Sigma$ is a focus-focus type if both systems $Z^{+}$ and $Z^{-}$ have an equilibrium point at $p$ with complex eigenvalues. A point $p \in \Sigma$ is a focus-parabolic type if:
the system defined in $Z^{+}$ has an equilibrium point of focus type at $p$ while the system $Z^{-}$ has a fold point at $p$, the case parabolic-focus is analogous. Finally, we say that $p \in \Sigma$ is a parabolic-parabolic type if the systems $Z^{+}$ and $Z^{-}$ have a fold point at $p$.

\begin{center}
	\begin{figure}[H]
		\begin{overpic}[scale=0.45]
			{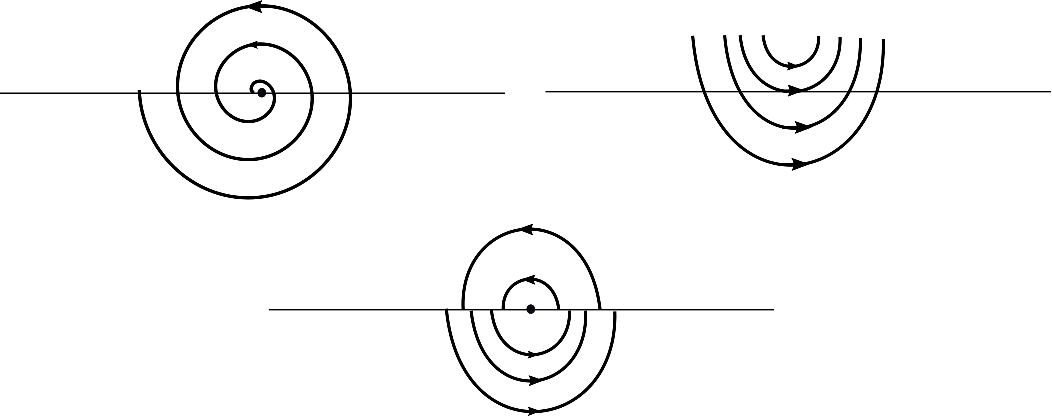} 
			\put(23,16){\textbf{a)}}
			\put(75,16){\textbf{b)}}
			\put(50,-5){\textbf{c)}}
		\end{overpic}
		\vspace{0.7cm}
		\caption{In picture \textbf{a)} we have a focus-focus point, in turn \textbf{b)} is an example of  parabolic-parabolic point and,  in picture $\textbf{c)}$ we have a focus-parabolic point.}\label{fold}
	\end{figure}
\end{center}

\vspace{-1.2cm}

Clearly, the piecewise-smooth system \eqref{eqpie} inherits the equilibria of $Z^{+}$ and $Z^{-}$. However, the equilibrium points may or may not appear in the phase portrait, depending on their location in the plane. When they appear, we say they are real, when they don't, we say they are virtual. We can also pay some attention to other equilibrium points of equilibria of the system \eqref{eqpie}. It is not difficult to see that $p$, being a sliding or escaping point with $Z^{+}(p)$ and $Z^{-}$ vectors dependent linearly, vanishes $Z^{s}(p)$ in \eqref{eqcampodeslizante}. These kinds of points are pseudo-equilibrium points and are internal points of a sliding segment, see \cite{Fil1988,GatManRin1973}. And, an equilibrium of the vector fields $Z^{\pm}$ belonging to $\Sigma$, is called a boundary equilibrium.

In what follows in this section, we will turn our attention to analyzing the non-smooth version of Welander's system \eqref{wel1} considering the non smooth function  \eqref{funck}.

Then, observe that the splitting manifold is given by $\{\rho=-\alpha T+ S=\epsilon\}.$ Making the following change of coordinates 
\begin{equation}\label{mudcoord}
	x=T, \qquad y  = \rho - \varepsilon= S-\alpha T-\epsilon,
\end{equation}
we obtain the system
\begin{equation}\label{eqw}
		\dot{x}  = 1-x-k(y) \, x,\qquad \dot{y} =  \beta-\beta \epsilon-k(y) \, \epsilon-\alpha - [\beta+k(y)] \, y-[\alpha \beta-\alpha] \, x,
\end{equation}
such that, under the coordinate change $k(y)	=k(y+\varepsilon)$ with
\begin{equation*}
	\begin{aligned}
k(y+\varepsilon)= 	k_1 , \ \text{if} \ y>0, \qquad k(y+\varepsilon)= 	k_0, \ \text{if} \ y<0.
%\left\{ \begin{array}{l} 
%	k_1 , \ \text{if} \ y>0,\\
%	k_0,   \ \text{if} \ y\leq0.
%	\end{array} \right.
	\end{aligned}
\end{equation*}
Now, the splitting manifold becomes $\{y = 0\}$. To make the computation easier, let us do a change of coordinates $(x,y)\rightarrow (y,x)$, which now the discontinuity manifold is given by $y-axis$. We will denote by $Z^{L}$ the vector field in the left side of the $y-axis$ and by $Z^{R}$ the vector field in the right side of the $y-axis$, where in our case, we obtain the following linear  system:
%\begin{equation}\label{eqw2}
%	\begin{aligned}
%		Z(x,y) = \left\{ \begin{array}{rll}
%			Z^{L}(x,y) & = & \left(- \frac{x}{2} + \frac{2y}{5}-\frac{\epsilon}{2} - \frac{3}{10} ,-y + 1\right),\vspace{0.2cm} \\ 
%			Z^{R}(x,y) & = & \left(-\frac{3x}{2} + \frac{2y}{5}-\frac{3\epsilon}{2} - \frac{3}{10},-2y + 1 \right),
%		\end{array} \right.
%	\end{aligned}
%\end{equation}
\begin{equation}\label{eqw2}
	\begin{aligned}
			&  Z^{L}(x,y)  = \left(
			\begin{array}{cc}
				-k_0-\beta & \alpha \, [1-\beta]  \\
				0    & -1 - k_0
			\end{array} \right)
			\left(
			\begin{array}{c}
				x \\
				y
			\end{array} \right)
			+
			\left(
				\begin{array}{c}
				-\alpha +\beta-[k_0+1] \, \varepsilon \\
				1
			\end{array} \right) ,\\
			\; \\
			&  Z^{R}(x,y) = \left(
		\begin{array}{cc}
			-k_1-\beta & \alpha \, [1-\beta]  \\
			0    & -1 - k_1
		\end{array} \right)
		\left(
		\begin{array}{c}
			x \\
			y
		\end{array} \right)
		+
		\left(
		\begin{array}{c}
			-\alpha +\beta-[k_1+1] \, \varepsilon \\
			1
		\end{array} \right).\\
	\end{aligned}
\end{equation}
\begin{center}
	\begin{figure}[H]
		\begin{overpic}[scale=0.37]
			{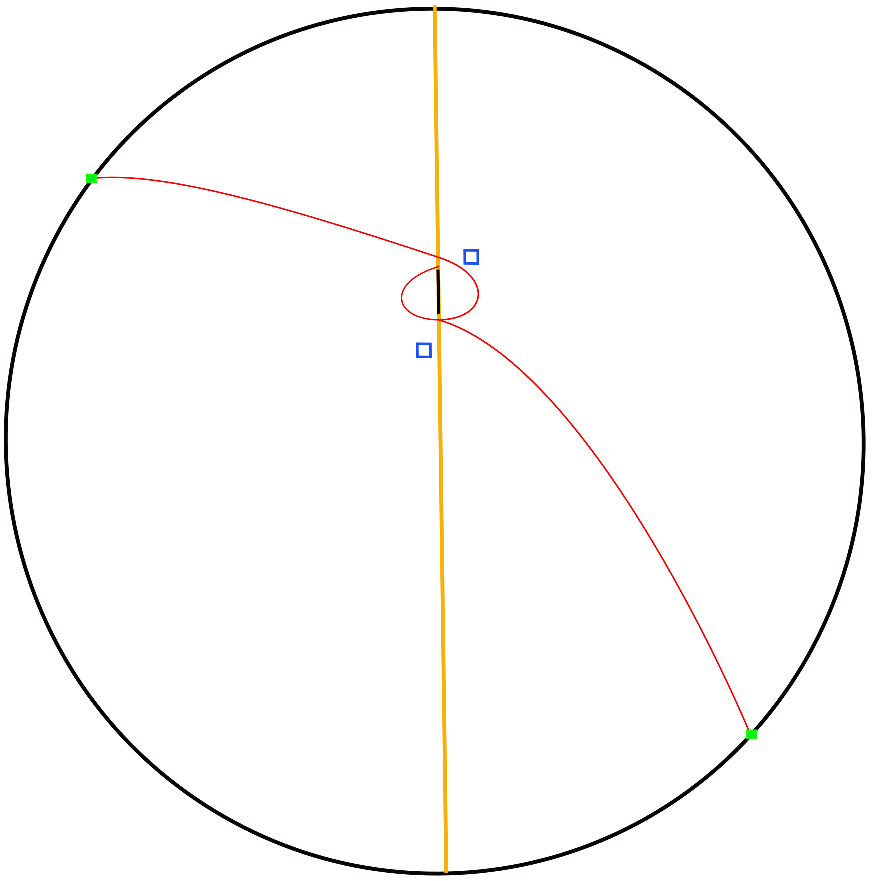} 
			\put(55,75){\textbf{$p_{2}$}}
			\put(40,60){\textbf{$p_{1}$}}
			\put(45,2){\textbf{$\Sigma$}}
		\end{overpic}
		\vspace{0.5cm}
		\caption{Phase portrait of non-smooth Welander system on Poincaré Disk. The discontinuity manifold is $y-axis$ for $\alpha=4/5$, $\beta=1/2$. The points $p_{1}$ and $p_{2}$  are invisible attracting node points for $Z^{R}$ and $Z^{L}$ respectively. The black segment represents a sliding region. The orange segment represents the $\Sigma$ manifold.}\label{welpie1}
	\end{figure}
\end{center}

The classical theory of linear differential systems shows that such types of systems do not have limit cycles. In this way, the only alternative for the existence of isolated periodic orbits is restricted to orbits that have some point in the discontinuity set.

As already observed, in planar piecewise smooth differential systems there are two kinds of limit cycles:  {\it sliding limit cycles} and {\it crossing limit cycles}. Sliding limit cycles contain some segment of the discontinuity set, and crossing limit cycles only contain some points of the discontinuity set, which the associated orbit crosses the discontinuity set, in these points.

Here, we are interested only in crossing limit cycles. Then, in order to firsts analyzes, let us assume $\mathrm{x}=(x,y)^{T} \in \mathbb{R}^2$, $A^{L}=(a_{ij}^{L})$ and $A^{R}=(a_{ij}^{R})$ are $2\times2$ constants matrices, and $b^{R}=(b_{1}^{R},b_{2}^{R})$ and $b^{L}=(b_{1}^{R},b_{2}^{L})$ are constant vectors of $\mathbb{R}^2$. Omitting the indexes $L,$ and $R$, we can write the linear piecewise system as 
$$\dot{\mathrm{x}}=A \cdot \mathrm{x}+b.$$
When we consider the discontinuity set being y-axis then we are considering the vector field with index $L$ for $x<0$ and the vector field with index $R$ for $x>0$. Analyzing the Lie derivative of each field on the y-axis, we see that the regions are determined by the product of $a_{12}^{L}y+b_1^{L}$ and $a_{12}^{R}y+b_1^{R}$.

Since our object of study is the \eqref{eqw2} system we have $a_{12}^{L}=a_{12}^{R}=\alpha \, (1-\beta)$. If $\alpha \, (1-\beta)=0$, that is, $\alpha=0$ or $\beta=1$, then the first coordinate, of the vectors fields, has always the same sign on the line {x=0}, making it impossible to create limit cycles crossing the discontinuity line.  Furthermore, simple calculations show us that the linear systems in \eqref{eqw2} can only have an equilibrium of attractor node type, remember that $\beta>0>-k_1$. In summary, we have just proved the following proposition.

\begin{proposition}\label{prop3}
	If $\alpha \, (1-\beta)=0$ then the piecewise smooth Welander's system \eqref{eqw2} does not have periodic orbits.
\end{proposition}

%Regarding linear systems, we can approach using another change of coordinates given by \cite{Freire1998, Ponce} where is possible to write the linear system in terms of the determinant and the trace of the matrix. 

Assuming $a_{12}^{L}=a_{12}^{R}=\alpha \, (1-\beta) \neq 0$, we can regard another change of coordinates given by \cite{Freire1998, Ponce}, where is possible to write the linear system in terms of the determinant and the trace of the matrix, and which is described in the next proposition.

\begin{proposition}[\cite{Ponce}]\label{prop1}
	Assuming the condition $a_{12}^{R} \, a_{12}^{L}>0$, the homeomorphism $\tilde{\mathrm{x}}=h(\mathrm{x})$ given by, 
	%if $\mathrm{x}<0,$
	\begin{equation*}\label{eqwnormal1}
		\tilde{\mathrm{x}} = 
		\begin{pmatrix}
			1 & 0  \\ 
			a_{22}^{L} & -a_{12}^{L}
		\end{pmatrix}\mathrm{x} - \begin{pmatrix}
			0 \\ 
			b_{1}^{L} 
		\end{pmatrix}, \,   \mathrm{x}<0,
%	\end{equation*}
%	\begin{equation*}\label{eqwnormal2}
	\; \; \tilde{\mathrm{x}} = 
		\dfrac{1}{a_{12}^{R}} \, \begin{pmatrix}
			a_{12}^{L} & 0  \\ 
			a_{12}^{L} \, a_{22}^{R} & -a_{12}^{L} \, a_{12}^{R}
		\end{pmatrix}\mathrm{x} - \begin{pmatrix}
			0 \\ 
			b_{1}^{L} 
		\end{pmatrix},  \, \mathrm{x}>0,
	\end{equation*}
%	if $\mathrm{x}>0$, 
	transform system $\dot{\mathrm{x}}=A \cdot \mathrm{x}+b$ into the canonical form, 
	\begin{equation*}\label{eqwnormal3}
		\dot{\mathrm{x}} = 
		\begin{pmatrix}
			tr(A^{L}) & -1  \\ 
			det(A^{L}) & 0
		\end{pmatrix}\mathrm{x} - \begin{pmatrix}
			0 \\ 
			a^{L} 
		\end{pmatrix},  \, x<0, \quad
%	\end{equation*}
%	if $\mathrm{x}<0$ and 
%	\begin{equation*}\label{eqwnormal4}
		\dot{\mathrm{x}} = 
		\begin{pmatrix}
			tr(A^{R}) & -1  \\ 
			det(A^{R}) & 0
		\end{pmatrix}\mathrm{x} - \begin{pmatrix}
			-b \\ 
			a^{R} 
		\end{pmatrix}, \, x>0,
	\end{equation*}
where $a^{L}=a_{12}^{L} \, b_{2}^{L}-a_{22}^{L} \, b_{1}^{L}$, $b=\frac{a_{12}^{L}}{a_{12}^{R}} \, b_{1}^{R}-b_{1}^{L}$, $a^{R}=\frac{a_{12}^{L}}{a_{12}^{R}} \, (a_{12}^{R} \, b_{2}^{R}-a_{22}^{R} \, b_{1}^{R})$.
\end{proposition}

\begin{remark}
	We observe that in the general case, the previous proposition reduces the number of parameters of the system from twelve to seven. Then for the next analyses, it is easier to study the piecewise system given by it.
%	We observe that in the general case, the previous proposition reduces the number of parameters of the system from twelve to five. Then, to study the Poincaré Map, it is easier to study the piecewise system given by \eqref{eqwnormal3} and \eqref{eqwnormal4}.
\end{remark}

%In this paper, we will choose the parameters $\alpha$ and $\beta$ that Welander used in his analysis for the smooth model, that is, we will consider $\alpha=\dfrac{4}{5}$ and $\beta=\dfrac{1}{2}$. 

Omitting the indexes $L,$ and $R$, the systems in Proposition \ref{prop1} can be written as 
\begin{equation}\label{sistemalinear}
	\dot{\mathrm{x}}=A\mathrm{x}-b,
%\end{equation}
%%with
%\begin{equation}
\qquad A=\begin{pmatrix}
		tr(A) & -1  \\ 
		det(A) & 0
	\end{pmatrix},
	\quad
	b=\begin{pmatrix}
		b_1 \\ 
		b_2
	\end{pmatrix}.	
\end{equation}
Now consider $[tr(A)]^2>det(A)>0$ and denote $\Delta=[tr(A)]^2-det(A)>0$. Then, considering $\lambda_1$, $\lambda_2$ the eigenvalues of the matrix $A$ with $\xi_1$, $\xi_2$ the associated eigenvectors then we can have
\begin{equation}
	\lambda_1=\dfrac{tr(A)+\sqrt{\Delta}}{2}, \quad
	\lambda_2=\dfrac{tr(A)-\sqrt{\Delta}}{2}, 
	\qquad 
	\xi_1=\begin{pmatrix}
		1 \\ 
		\lambda_2
	\end{pmatrix},
	\quad 
	\xi_2=\begin{pmatrix}
		1 \\ 
		\lambda_1
	\end{pmatrix},
\end{equation}
such that, clearly, $\lambda_1>\lambda_2.$
The equilibrium point and its invariant manifolds are 
%given by
\begin{equation*}
	\mathrm{x}_e=\dfrac{1}{det(A)}\begin{pmatrix}
		b_2 \\ 
		tr(A) \, b_2-det(A) \, b_1
	\end{pmatrix},
	\;
\ell_1: \, y=\lambda_2 \, x+\frac{b_2}{\lambda_2}-b_1, \; \ell_2: \, y=\lambda_1 \, x+\frac{b_2}{\lambda_1}-b_1.
\end{equation*}

In this way, for $\alpha \, (1-\beta) \neq 0$, using the Proposition \ref{prop1} we can write \eqref{eqw2} as the following
\begin{equation}\label{eqw3}
	\begin{aligned}
		&  Z^{L}(x,y)  = \left(
		\begin{array}{rr}
			-(1+2 \, k_0+\beta) & -1  \\
			(1+k_0)  (k_0+\beta)   &  0
		\end{array} \right)
		\left(
		\begin{array}{r}
			x \\
			y
		\end{array} \right)
		-
		\left(
		\begin{array}{c}
		0 \\
		a^L
		\end{array} \right)  \; \mbox{if}, \; x<0,\\
		\; \\
		&  Z^{R}(x,y) = \left(
		\begin{array}{rr}
		-(1+2 \, k_1+\beta) & -1  \\
			(1+k_1)  (k_1+\beta)   &  0
		\end{array} \right)
		\left(
		\begin{array}{r}
			x \\
			y
		\end{array} \right)
		-
		\left(
		\begin{array}{c}
		-(k_0-k_1) \, \varepsilon  \\
			a^R
		\end{array} \right) \, \mbox{if} \,x>0, \\
	\end{aligned}
\end{equation}
such that  $\Sigma=\{x=0\}$ and
\[
\begin{aligned}
	& 	a^L =- \alpha \, (k_0+ \beta) -(1+k_0)(\beta\, (\varepsilon-1)+k_0 \, \varepsilon),\\
& a^R= - \alpha \, (k_1+ \beta) -(1+k_1)(\beta \, (\varepsilon-1)+k_1 \, \varepsilon).
	\end{aligned}
	\]
%\begin{equation}\label{eqw1}
%	\begin{aligned}
%		Z(x,y) = \left\{ \begin{array}{rll}
%			Z^{+}(x,y) & = & \left(-2x + 1,-\dfrac{3}{10} - \dfrac{3}{2}\epsilon - \dfrac{3}{2}y + \dfrac{2}{5}x\right),\vspace{0.2cm} \\ 
%			Z^{-}(x,y) & = & \left(-x + 1,-\dfrac{3}{10} - \dfrac{\epsilon}{2} - \dfrac{y}{2} + \dfrac{2}{5}x\right),
%		\end{array} \right.
%	\end{aligned}
%\end{equation}
Clearly, the vector field on the left $Z^L$ has the eigenvalues $\lambda_1^L=-k_0 -\beta<0$ and $\lambda_2^L=-k_0-1<0$. On the other hand, the vector field on the right $Z^R$ has the eigenvalues $\lambda_1^R=-k_1-\beta<0$ and $\lambda_2^R=-k_1-1<0$. In the sequence, we will consider
\begin{equation}\label{alphaLR}
	\begin{aligned}
	&\alpha^L_{\beta,\varepsilon}=\dfrac{-(1+k_0)(\beta \, (\varepsilon-1)+k_0 \, \varepsilon)}{k_0+ \beta}, \quad \alpha^R_{\beta,\varepsilon}=\dfrac{-(1+k_1)(\beta \, (\varepsilon-1)+k_1 \, \varepsilon)}{k_1+ \beta},\\
	&  \varepsilon^*_\beta=\dfrac{\beta \, (\beta-1)}{(k_0+\beta)(\beta+k_1)}
\end{aligned}
\end{equation}
where $\alpha^R_{\beta,\varepsilon}>\alpha^L_{\beta,\varepsilon}$, $\alpha^R_{\beta,\varepsilon}=\alpha^L_{\beta,\varepsilon}$ and $\alpha^R_{\beta,\varepsilon}<\alpha^L_{\beta,\varepsilon}$, respectively, if $\varepsilon<\varepsilon^*_\beta$, $\varepsilon=\varepsilon^*_\beta$ and $\varepsilon>\varepsilon^*_\beta$. Then, analyzing by direct calculations the equilibrium points, the Lie derivative of the system linear \eqref{eqw3}, and some analysis of signs,  we have the following results.

\begin{proposition}\label{propeqdob}
	Consider the system \eqref{eqw3}  and  $\alpha^L_{\beta,\varepsilon}$, $\alpha^L_{\beta,\varepsilon}$ in \eqref{alphaLR}, then following statements hold.
	\begin{itemize}
\item[(i)]	For $Z^{L}(x,y)$ we have that:

\vspace{0.15cm}

\begin{itemize}
	\item[($i_1$)] there is an equilibrium point of the attractor node type
	\begin{equation}
\left(
	\begin{array}{c}
		x_e^L \\
		\; \\
		y_e^L
	\end{array} \right)
	=
\dfrac{1}{(1+k_0)  (k_0+\beta) } \left( \begin{array}{c}
	a^L\\
	\; \\
	-(1+2 \, k_0+\beta) \, a^L
\end{array} \right),
\phantom{-
	\left( \begin{array}{c}
		0 \\
		\; \\
		(k_1-k_0)\varepsilon
	\end{array} \right)}
\end{equation} 
\end{itemize}
such that $\mathrm{x^{L}_e}=(x_e^L,y_e^L)^{T}$ is an equilibrium point real, virtual, or of boundary, respectively if, $\alpha > \alpha^L_{\beta,\varepsilon}$, $\alpha < \alpha^L_{\beta,\varepsilon}$ or $\alpha = \alpha^L_{\beta,\varepsilon}$, with the boundary equilibrium point occurring at the origin;

\vspace{0.15cm}

\begin{itemize}
	\item[($i_2$)] the equilibrium point $\mathrm{x^{L}_e}$  has the following invariant manifolds
	\begin{equation}\label{imanifoldL}
	\ell_1^L: \, y=-(k_0+1) \, x-\frac{a^{L}}{(k_0+1)}, \; \ell_2^L: \; y=-(k_0+\beta) \, x-\frac{a^{L}}{(k_0+\beta)};
	\end{equation}
	%	at origin such that

\end{itemize}	

\vspace{0.15cm}

\begin{itemize}
	\item[($i_3$)] there is a fold point of order two in  $\mathrm{x^{L}_t}=(x_t^L,y_t^L)^{T}=(0,0)^{T}$ such that the fold point is visible if $\alpha > \alpha^L_{\beta,\varepsilon}$ and invisible if  $\alpha < \alpha^L_{\beta,\varepsilon}$.
%	at origin such that

\end{itemize}	

\vspace{0.15cm}

\item[(ii)]	For $Z^{R}(x,y)$ we have that:

\vspace{0.15cm}

\begin{itemize}
\item[($ii_1$)]  there is an equilibrium point of  the attractor node type

	\begin{equation}
		 \left(
\begin{array}{c}
	x_e^R \\
	\; \\
	y_e^R
\end{array} \right)
=
\dfrac{1}{(1+k_1)  (k_1+\beta) } \left( \begin{array}{c}
a^R \\
	\; \\
	-(1+2 \, k_1+\beta) \, a^R
\end{array} \right)
-
\left( \begin{array}{c}
	0 \\
	\; \\
	-(k_0-k_1) \, \varepsilon
\end{array} \right),
\end{equation} 
\end{itemize}
such that $\mathrm{x^{R}_e}=(x_e^R,y_e^R)^{T}$ is an equilibrium point real, virtual, or of boundary, respectively if, $\alpha < \alpha^R_{\beta,\varepsilon}$, $\alpha > \alpha^R_{\beta,\varepsilon}$ or $\alpha = \alpha^R_{\beta,\varepsilon}$, with the boundary equilibrium point occurring at $(0,(k_0-k_1) \, \varepsilon)^{T}$;

\vspace{0.15cm}

\begin{itemize}
	\item[($i_2$)] for $b_1^{R}=-(k_0-k_1) \, \varepsilon$, the equilibrium point $\mathrm{x^{L}_e}$  has the following invariant manifolds
	\begin{equation}\label{imanifoldR}
	\ell_1^R: \, y=-(k_1+1) \, x-\frac{a^{R}}{(k_1+1)} - b_1^{R},  \; \ell_2^R: \, y=-(k_1+\beta) \, x-\frac{a^{R}}{(k_1+\beta)}- b_1^{R} ;
	\end{equation}
	%	at origin such that

\end{itemize}	

\vspace{0.15cm}

\begin{itemize}
	\item[($ii_3$)] there is a fold point of order two in $$\mathrm{x^{R}_t}=(x_t^R,y_t^R)^{T}=(0,(k_0-k_1) \, \varepsilon)^{T}= (0,-b_1^{R})^{T}$$
	 such that the fold point is visible if $\alpha < \alpha^R_{\beta,\varepsilon}$ and invisible if $\alpha > \alpha^R_{\beta,\varepsilon}$.
	%	at origin such that	
\end{itemize}	
\end{itemize}

\vspace{0.15cm}

	In addition, for $\varepsilon=0$ the fold points $\mathrm{x^{L}_t}$ and $\mathrm{x^{R}_t}$ coincide.
\end{proposition}

\begin{proposition}\label{prop7}
For the linear system \eqref{eqw3}  the next statements are valid in the discontinuity line $\Sigma=\{x=0\}$.
\begin{itemize}
	\item [(i)] If $\varepsilon<0$ then $0=y_t^L<y_t^R=(k_0-k_1) \, \varepsilon$ such that $\Sigma \cap \{y<0\} $ is a positive crossing region, $\Sigma \cap \{0<y<y_t^R\} $ is an escaping region and, 
 $\Sigma \cap \{y_t^R<y\} $ is a negative crossing region.
%	\begin{itemize}
%		\item [.] $\Sigma \cap \{y<0\} $ is a positive crossing region, 
%		
%			\item [.] $\Sigma \cap \{0<y<y_t^R\} $ is an escaping region,
%			
%				\item [.] $\Sigma \cap \{y_t^R<y\} $ is a negative crossing region.
%	\end{itemize}

	\item [(ii)] If $\varepsilon=0$ then $0=y_t^L=y_t^R$ such that   $\Sigma \cap \{y<0\} $ is a positive crossing region and, $\Sigma \cap \{0<y\} $ is a negative crossing region.
%	\begin{itemize}
%		\item [.] $\Sigma \cap \{y<0\} $ is a positive crossing region,
%		
%		\item [.] $\Sigma \cap \{0<y\} $ is a negative crossing region.
%		
%		\vspace{0.15cm}
%		
%		
%	\end{itemize}

	\item [(iii)] If $\varepsilon>0$ then $(k_0-k_1) \, \varepsilon=y_t^R<y_t^L=0$ such that $\Sigma \cap \{y<y_t^R\} $ is a positive crossing region, $\Sigma \cap \{y_t^R<y<0\} $ is a sliding region and, $\Sigma \cap \{0<y\} $ is a negative crossing region.
%	\begin{itemize}
%		\item [.] $\Sigma \cap \{y<y_t^R\} $ is a positive crossing region,
%		
%		\item [.] $\Sigma \cap \{y_t^R<y<0\} $ is a sliding region,
%		
%		\item [.] $\Sigma \cap \{0<y\} $ is a negative crossing region,
%	\end{itemize}

\end{itemize}
\end{proposition}

The Proposition \ref{propeqdob} assures us that it is not possible for the appearance of a fold of order greater than two for the system \eqref{eqw3}. This is because for $\alpha=\alpha^L_{\beta,\varepsilon}$ or $\alpha=\alpha^R_{\beta,\varepsilon}$, values which nullify $(Z^{\pm})^2f$ at the respective bending point, cause a collapse between of equilibrium point and the fold point in a boundary equilibrium of the respective system. And so, $(Z^{\pm})^nf=0$ for all $n$, since we have an equilibrium of the system $Z^{\pm}$ on the discontinuity line, that is a boundary equilibrium. Moreover, the invariant manifolds \eqref{imanifoldL} and \eqref{imanifoldR} intersect the discontinuity set $\Sigma=\{x=0\}$, respectively, at the points $(0,y_{m_1}^{L,R})$ and $(0,y_{m_2}^{L,R}+(k_0-k_1) \, \varepsilon)$, where
\begin{equation}\label{intx=0}
y_{m_1}^{L}=-\dfrac{a^{L}}{(k_0+1)}, \; y_{m_2}^{L}=-\dfrac{a^{L}}{(k_0+\beta)}, \; y_{m_1}^{R}=-\dfrac{a^{R}}{(k_1+1)}, \;
y_{m_2}^{R}=-\dfrac{a^{R}}{(k_1+\beta)}.
\end{equation}

Now we will introduce a new auxiliary function that will help us to study the periodic orbits of the piecewise system given by  Proposition \ref{prop1}:
\begin{equation}\label{eqpsi}
	\psi_{r_{1},r_{2}}(t)=r_{1}-r_{2}+r_{2} \, e^{r_{1}t}-r_{1} \, e^{r_{2}t},
\end{equation}
where $t$ is variable and $r_{1},r_{2}$ are parameters. Then we have the following lemma with proof elementary.

\begin{lemma}\cite{node}\label{lemma}
	For $r_{1}$, $r_{2}$, $t \in \mathbb{R}$, $\psi_{r_{1},r_{2}}(t)$  \eqref{eqpsi} has the following properties:
	\begin{itemize}
		\item[(1)] $\psi_{0,r_{2}}(t)=\psi_{r_{1},0}(t)=\psi_{r_{1},r_{2}}(0)=0$. $\;$ \rm{(2)} $\psi_{-r_{1},-r_{2}}(t)=-\psi_{r_{1},r_{2}}(-t)$.
			\item[(3)] $\psi_{qr_{1},qr_{2}}(t)=q\psi_{r_{1},r_{2}}(kt),$ $ \forall \; q\in \mathbb{R}$. $\; \; \; \, $  \rm{(4)} $\psi_{r_{1},r_{2}}(t)=-\psi_{r_{2},r_{1}}(t)$
			\item[(5)] When $r_{1}>r_{2}$ and $r_{1} \, r_{2}>0$, then
		\begin{equation*}
			\begin{aligned}
		&	\lim_{t\rightarrow +\infty} \psi_{r_{1},r_{2}}(t) = \left\{\begin{array}{ll}
				+\infty, & r_{1}> r_{2}>0; \\ 		
				r_{1}-r_{2} , & 0>r_{1}> r_{2};
			\end{array}\right. \\
			&\lim_{t\rightarrow -\infty} \psi_{r_{1},r_{2}}(t) = \left\{\begin{array}{ll}
				r_{1}-r_{2}, & r_{1}> r_{2}>0; \\ 		
				+\infty, & 0>r_{1}> r_{2};
			\end{array}\right.
			\end{aligned}
		\end{equation*}
			\item[(6)] When $r_{1} \, r_{2}>0$ we have $\psi_{r_{1},r_{2}}(t)>0$ for any $t\neq0$.
	\end{itemize}	
\end{lemma}
 
 \begin{remark}\label{obsphi}
 	Note that the proof of statement (6) of Lemma \ref{lemma} is given by statement (5) plus the verification that the auxiliary function $\psi_{r_{1},r_{2}}(\pm t)$ in \eqref{eqpsi} has only one critical point at $t=0$ and a unique inflection point at  
 	$$t_{inf}^{\pm}=\mp \frac{\ln(\frac{r_1}{r_2})}{(r_1-r_2)}.$$
 	Which $t_{inf}^{+}$ is positive for $0>r_1>r_2$ and negative for $0<r_2<r_1$, on the other side, $t_{inf}^{-}$ is negative for $0>r_1>r_2$ and positive for $0<r_2<r_1$. Moreover,  $\psi_{r_{1},r_{2}}^\prime(\pm t) =\mp r_1r_2 \, \phi_{r_{2},r_{1}}(\pm t)$ with $\phi_{r_{2},r_{1}}(\pm t)=e^{\pm r_{2}t} -e^{\pm r_{1}t}$ having a unique root at $t=0$ and satisfying 
 	\begin{equation*}
 		\begin{aligned}
 		\lim_{t\rightarrow +\infty} \phi_{r_{2},r_{1}}(t)= \left\{\begin{array}{ll}
 				- \infty, & r_{1}> r_{2}>0 \\ 		
 				0 ,& 0>r_{1}> r_{2}
 			\end{array}\right.;  \;  \;
 			\lim_{t\rightarrow -\infty} \phi_{r_{2},r_{1}}(t)= \left\{\begin{array}{ll}
 				0, & r_{1}> r_{2}>0\\ 		
 				+ \infty, & 0>r_{1}> r_{2}
 			\end{array}\right..
 		\end{aligned}
 	\end{equation*}
 	Consequently $\phi_{r_{2},r_{1}}(t_{inf}^+)$ must be negative for $0>r_1>r_2$ and positive for $0<r_2<r_1$, that is, $\phi_{r_{2},r_{1}}(t)$ is always negative for $t>0$  with $0>r_1>r_2$ and always positive for  $t<0$ with $0<r_2<r_1$. And, when $r_{1}\,  r_{2}>0$ follows that $\psi_{r_{1},r_{2}}(-t)>0$ for any $t\neq0$.
 \end{remark}
 
Turning our attention to the linear system \eqref{sistemalinear} then its solution starting at $\mathrm{x}=(x_0,y_0)^T$, is given by

\begin{equation}\label{sollinear}
	\begin{aligned}
		\begin{pmatrix}
			x(t) \\ 
			\; \\
			y(t)
		\end{pmatrix}
		&=
		\dfrac{1}{\lambda_1-\lambda_2} \left( \begin{array}{l}
			(\lambda_1 e^{\lambda_1t}- \lambda_2 e^{\lambda_2t}) \, x_0+( e^{\lambda_2t}- e^{\lambda_1t} ) \, y_0\\ 
			\; \\
			det(A)( e^{\lambda_1t}- e^{\lambda_2t} ) \, x_0 + 	(\lambda_1 e^{\lambda_2t}- \lambda_2 e^{\lambda_1t}) \, y_0
		\end{array} \right)
		+ \\
		\; \\
		& +\dfrac{1}{\lambda_1-\lambda_2}  \left( \begin{array}{l}
			\dfrac{	\psi_{\lambda_{1},\lambda_{2}}(t)}{det(A)} \, b_2- ( e^{\lambda_1t}- e^{\lambda_2t} ) \, b_1\\ 
			\; \\
			\dfrac{[ \lambda_1^2(1-e^{\lambda_2t})-	\lambda_2^2(1-e^{\lambda_1t})]}{det(A)} \, b_2-	\psi_{\lambda_{1},\lambda_{2}}(t) \, b_1
		\end{array} \right),
	\end{aligned}
\end{equation}
where $\lambda_1$ and $\lambda_2$ correspond to eigenvalues to the matrix $A$ such that $\lambda_1>\lambda_2$ and $(tr(A))^2>det(A)>0.$ And,
%Finally, using the index, we can define the Poincaré Map for the Welander's system \eqref{eqw2}
in order to apply the Lemma \ref{lemma} in the sequence, we need to fix a order relationship between the eigenvalues of Welander's system \eqref{eqw3}.

\begin{itemize}
	\item [(I)] The case $0<\beta<1$ guarantees that  $0>\lambda_1^L>\lambda_2^L$ and $0>\lambda_1^R>\lambda_2^R$. Then we consider the auxiliary function \eqref{eqpsi}, in the respective region, as
	\begin{equation}\label{eqphiL}
		\begin{aligned}
			&\psi_{\lambda_{1}^L,\lambda_{2}^L}^L(t)=1-\beta-(k_0+1) \, e^{-(k_0+\beta)t}+(k_0+\beta) \, e^{-(k_0+1)t}, \; \; \mbox{for} \; Z^L,\\
			&\psi_{\lambda_{1}^R,\lambda_{2}^R}^R(t)=1-\beta-(k_1+1) \, e^{-(k_1+\beta)t}+(k_1+\beta) \, e^{-(k_1+1)t}, \; \; \mbox{for} \; Z^R.
		\end{aligned}
	\end{equation}
	
	\item [(II)] In turn,  when  $1<\beta$ we have  $\lambda_1^L<\lambda_2^L<0$ and $\lambda_1^R<\lambda_2^R<0$. Then we consider the auxiliary function \eqref{eqpsi}, in each region, as
	\begin{equation}\label{eqphiR}
		\begin{aligned}
			&\psi_{\lambda_{2}^L,\lambda_{1}^L}^L(t)=\beta-1-(k_0+\beta) \, e^{-(k_0+1)t}+(k_0+1) \, e^{-(k_0+\beta)t}, \; \; \mbox{for} \; Z^L,\\
			&\psi_{\lambda_{2}^R,\lambda_{1}^R}^R(t)=\beta-1-(k_1+\beta) \, e^{-(k_1+1)t}+(k_1+1) \, e^{-(k_1+\beta)t}, \; \; \mbox{for} \; Z^R.
		\end{aligned}
	\end{equation}
	
\end{itemize}
%We will introduce now a new auxiliary functions that will help us to study the Poincaré Map of the piecewise system given by \eqref{eqwnormal3} and \eqref{eqwnormal4}. 
%Omitting the indexes $L,$ and $R$, let us consider $\lambda_{1}$ and $\lambda_{2}$ the eigenvalues of \eqref{eqwnormal3} and \eqref{eqwnormal4} and let us consider $\psi_{\lambda_{1},\lambda_{2}}(t)=\lambda_{1}-\lambda_{2}+\lambda_{2}e^{\lambda_{1}t}-\lambda_{1}e^{\lambda_{2}t}$.
Summarizing, we must suitably consider the auxiliary function in each one of the cases, where we will always have $\lambda_{i}^{L,R}>\lambda_{j}^{L,R}$, $i,j=1,2$ with $i\neq j$, and  Lemma \ref{lemma} is valid. Furthermore, it is this adequate auxiliary function that must be considered in the solution of the \eqref{sollinear} system in the case of the Welander system \eqref{eqw3}.

Continuing the analysis of the periodic orbits of the system \eqref{eqw3}, we have $$(tr(A^{L,R}))^2>4det(A^{L,R})>0,  \mbox{with} \; tr(A^{L,R})<0,$$ 
and, then Huan et al in \cite[ Prop. 2.2]{node} show that both attractor nodes being virtual is a necessary condition for the existence of periodic orbits of \eqref{eqw3}. 

Denoting by $Z^{L,R}(x,y)=(Z^{L,R}_1(x,y), Z^{L,R}_2(x,y))$ the
left and right vector field of  Welander system \eqref{eqw3}, then we have directly that 
$$Z^L_1(0,y)=-y, \qquad \mbox{and} \qquad Z^R_1(0,y)=-y+(k_0-k_1) \, \varepsilon.$$
Therefore, under the action of $Z^L$ ($Z^R$), a solution starting from $(0,y_0)$ with $y_0>0$ ($y<(k_0-k_1)\varepsilon$) goes into the left zone $\{x<0\}$ (right zone $\{x>0\}$) and then it can stay in this region or leave the left zone (right zone) through the discontinuity line $\{x=0\}$ at a point $(0,y)$ with $y<0$ ($y>(k_0-k_1) \, \varepsilon$) after a positive finite time.  Therefore, the existence of a periodic orbit of the Welander system intersecting the discontinuity  line \{x = 0\} at two points, $(0, y_0)$ and $(0, y_1)$ with $y_0 > y_1$, satisfy
\begin{equation}\label{condy0y1}
	y_0>max\{0,(k_0-k_1) \, \varepsilon\}, \qquad y_1<min\{0,(k_0-k_1) \, \varepsilon\}.
\end{equation}
Because otherwise the orbits can't connect in the right direction in a crossing region, and so, to form the desired orbit. For the left system $Z^L$, considering an initial condition $(x_0,y_0) ^T=(0,y_0)$ with \eqref{condy0y1} being valid, then using \eqref{sollinear} its solution in the future cross the discontinuity line when
\begin{equation}\label{condopL}
	0=
	\frac{1}{\lambda_i^{L}-\lambda_j^{L}} \left(
	( e^{\lambda_j^{L}t}- e^{\lambda_i^{L}t} ) \, y_0 + \frac{	\psi_{\lambda_i^{L},\lambda_j^{L}}(t) }{(1+k_0)  (k_0+\beta) } \, a^{L}\right), \quad 0>\lambda_{i}^{L}>\lambda_{j}^{L}.
\end{equation}
However if $a^{L} \leq 0$, that is, $ \alpha \geq \alpha^L_{\beta,\varepsilon}$, by statement (6) of Lemma \ref{lemma} and Remark \ref{obsphi} the equation \eqref{condopL} is impossible because the right-hand side is always negative. Similarly, for the right system $Z^R$, and  an initial condition $(x_0,y_0) ^T=(0,y_1)$ satisfying \eqref{condy0y1} by \eqref{sollinear}  its solution in future cross  the discontinuity if
\begin{equation}\label{condopR}
	0=
	\frac{1}{\lambda_i^{R}-\lambda_j^{R}} \left(
	( e^{\lambda_j^{R}t}- e^{\lambda_i^{R}t} ) \, (y_1-(k_0-k_1) \, \varepsilon) + \frac{	\psi_{\lambda_i^{L},\lambda_j^{R}}(t)}{(1+k_1)  (k_1+\beta) } \, a^{R}\right),
\end{equation}
such that $ 0>\lambda_{i}^{R}>\lambda_{j}^{R}$. Here if 
$a^{R} \geq 0$, that is, $ \alpha \leq \alpha^R_{\beta,\varepsilon}$, then by statement (6) of Lemma~\ref{lemma} and Remark~\ref{obsphi} the equation \eqref{condopR} is impossible because right-hand side is always positive. Therefore, observing these facts, and adapting to our case and notation, we have just proved the following propositions.

\begin{proposition}[\cite{node}]\label{propcondnec}
The conditions $\alpha < \alpha^L_{\beta,\varepsilon}$ and  $\alpha > \alpha^R_{\beta,\varepsilon}$ are necessary for the existence of periodic orbits of the system \eqref{eqw3}.
\end{proposition}

\begin{proposition}\label{prop11}
	If $\alpha \, (1-\beta) \neq 0$ such that $\alpha \geq \alpha^L_{\beta,\varepsilon}$ or  $\alpha \leq \alpha^R_{\beta,\varepsilon}$, then the piecewise smooth Welander's system \eqref{eqw3} does not have periodic orbits.
\end{proposition}
%Analyzing the equilibrium points $Z^{+}$ of \eqref{eqw1}, follow that $p_{1}=(x,y)=\left(\frac{1}{2},-\epsilon-\frac{1}{15}\right)$ is a invisible attracting node. On the other hand, $p_{2}=(x,y)=\left(1,-\epsilon+\frac{1}{5}\right)$ is also a invisible attracting node for $Z^{-}$.

%\begin{center}
%	\begin{figure}[H]
%		\begin{overpic}[scale=0.5]
%			{wel_pie.eps} 
%			\put(55,45){\textbf{$p_{1}$}}
%			\put(73,60){\textbf{$p_{2}$}}
%			\put(-3,49){\textbf{$\Sigma$}}
%			\end{overpic}
%		\vspace{1.0cm}
%		\caption{Phase portrait of non-smooth Welander system on Poincaré Disk. The points $p_{1}$ and $p_{2}$  attracting node points for $Z^{+}$ and $Z^{-}$ respectively. The red segment represents a sliding region. The orange segment represents the $\Sigma$ manifold.}\label{welpie}
%	\end{figure}
%\end{center}

%The system~\eqref{eqw1} has a sliding segment between the points $x=\frac{3}{4}+\frac{5}{4}\epsilon$ and $x=\frac{3}{4}+\frac{15}{4}\epsilon$. To obtain these points is enough to solve $\frac{1}{2}+\frac{1}{2}\epsilon-k \epsilon-\frac{4}{5}+\frac{2}{5}x=0$ taking $k=0$ and $k=1$.

\section{Welander's Poincaré Map}\label{PMW}

In Section \ref{pre} it was possible to prove the non-existence of periodic orbits of the piecewise smooth Welander's system, for some parameter conditions. The remaining conditions are $\alpha < \alpha^L_{\beta,\varepsilon}$ and $\alpha > \alpha^R_{\beta,\varepsilon}$ which allow the existence of some kind of return in $\{x=0 \}$, for the piecewise smooth Welander's system \eqref{eqw3}. Therefore, for this case, the periodic orbits' existence should be analyzed using the Poincaré map. Then, in this section, we will study the Poincaré Map for the non-smooth Welander system \eqref{eqw3}.

As in the continuous case, the Poincaré Map (or first return map) is an essential tool for studying the stability and bifurcations of periodic orbits. This occurs due to the fixed points of the Poincaré Map corresponding to the number of periodic orbits. By studying the derivatives of the Poincaré Map, we can determine the stability of the periodic orbit. A simple way to define the Poincaré Map for the continuous case on the plane is considering a segment $\Sigma$ such that intersects the curve $\Gamma$ transversally at $p$, see Figure~\ref{pmc}.
	\begin{center}
		\vspace{0.31cm}
	\begin{figure}[H]
		\begin{overpic}[scale=0.42]
			{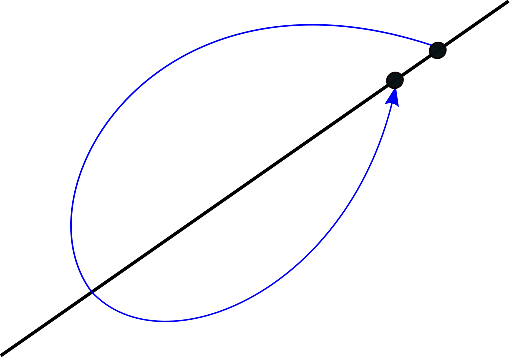} 
			\put(86,65){\textbf{$p$}}
			\put(80,45){\textbf{$\Pi(p)$}}
			\put(-5,3){\textbf{$\Sigma$}}
			\put(45,65){\textbf{$\Gamma$}}
		\end{overpic}
		\vspace{0.5cm}
		\caption{Poincaré Map for a continuous vector field on the plane.}\label{pmc}
	\end{figure}
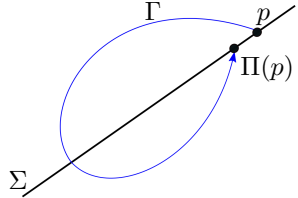
\end{center}

However, for a non-smooth vector field, we have, at least, two vector fields. Then, to study the Poincaré Map for a non-smooth ($\Pi_{d}$)vector field we have to study the Poincaré Map $\Pi^{R}$ of $Z^{R}$ and the Poincaré Map $\Pi^{L}$ of $Z^{L}$. Then the Poincaré Map will be given by $\Pi_{d}$=$\Pi^{R}\circ \Pi^{L}$. The fixed points of $\Pi_{d}$ can be obtained from the zeros of $\Pi_{d}(\rho)-\rho$. For simplicity, instead of this map, we will compute the equivalent one $\Delta(\rho)=(\Pi^{R})^{-1}(\rho)-\Pi^{L}(\rho)$, see Figure~\ref{pmd}.
\begin{center}
	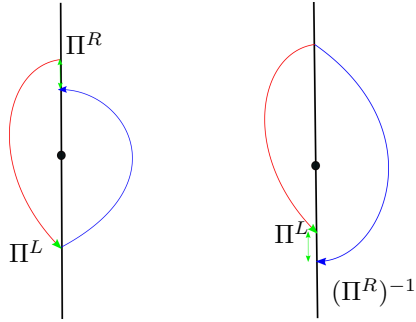
\begin{figure}[H]
		\begin{overpic}[scale=0.4]
			{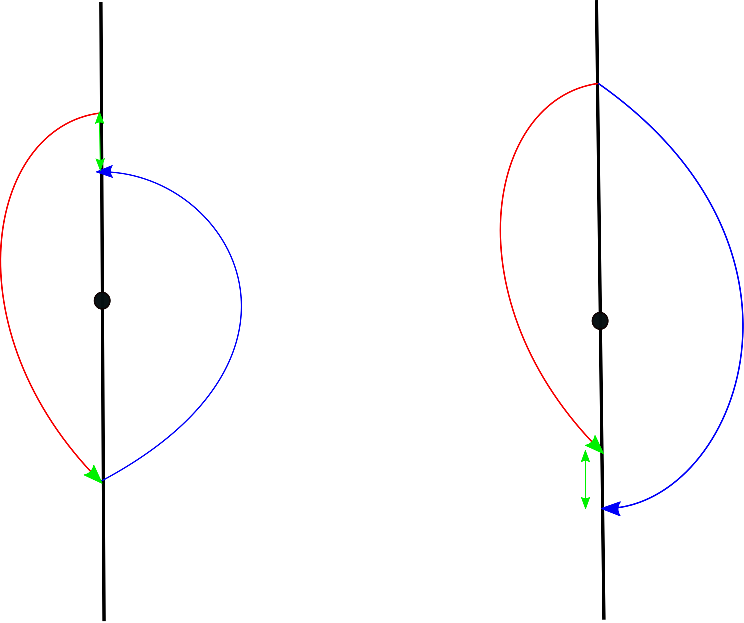} 
			\put(15,70){\textbf{$\Pi^{R}$}}
			\put(0,15){\textbf{$\Pi^{L}$}}
			\put(70,20){\textbf{$\Pi^{L}$}}
			\put(85,5){\textbf{$(\Pi^{R})^{-1}$}}
		\end{overpic}
		\vspace{0.4cm}
		\caption{Poincaré Map for a discontinuous vector field on the plane.}\label{pmd}
	\end{figure}
\end{center}

\vspace{-1.3cm}

In many cases, obtaining the exactly fixed points of the Poincaré Map is not possible. In addition, in most cases obtaining the exact Poincaré Map is not possible. Then, an alternative to studying the Poincaré Map is to study the coefficients of the Taylor series of the difference map $\Delta$. Here, the coefficients are called Lyapunov Constants. When we have a focus-focus point in the discontinuity line, a usual way to study the Poincaré Map is by doing a polar change of coordinate in the system $Z(x,y)$. On the other hand, the same approach can not be done for parabolic points once the polar change of coordinates does not guarantee the analyticity of the Poincaré Map like in the case that we have a focus. This gap can be contoured using the generalized polar coordinates, see \cite{GCP}. The $(R,\theta, p,q)$ generalized polar coordinates are given by $x=R^{p}Cs(\theta)$, $y=R^{q}Sn(\theta)$ where $p$ and $q$ will be fixed afterwards and where $Cs(\theta)$ and $Sn(\theta)$ are the solution of Cauchy problem,
\begin{equation*}
\begin{aligned}
& \dot{Cs}(\theta)=-Sn^{2p-1}(\theta), \qquad  
\dot{Sn}(\theta)=Cn^{2q-1}(\theta), \\
&Cs(0)=\sqrt[2q]{\frac{1}{p}}, \qquad  \qquad  \quad Sn(0)=0.
\end{aligned}
\end{equation*} 
The reader can see more details in \cite{SDB,GCP1,LiP}.

Despite such a change of coordinates guaranteeing the analyticity of the Poincaré Application, studying such an application implies working with many integrals that are not obvious to solve. In addition, this method does not work when we have a sliding segment. Because of these facts, we have another strategy for approaching our case. We are going to define the Poincaré Map for the Welander system \eqref{eqw3}, for this we need to understand each return map of the left and right sides.

Before proceeding, remember that we need to consider the auxiliary function \eqref{eqpsi} with the eigenvalues of the system \eqref{eqw3}. Moreover, the respective eigenvalues $\lambda_{i}^{L,R}$ and $\lambda_{j}^{L,R}$ must satisfy
$\lambda_{i}^{L,R}>\lambda_{j}^{L,R}$, $i,j=1,2$ with $ i\neq j$, and 

\begin{equation}\label{eqlambda}
	\begin{aligned}
		& \lambda_1^L=-k_0 -\beta<0, \qquad  \lambda_2^L=-k_0-1<0, \\ 
		& \lambda_1^R=-k_1-\beta<0, \qquad \lambda_2^R=-k_1-1<0,
	\end{aligned}
\end{equation}
see the equations \eqref{eqphiL} and \eqref{eqphiR} for more details.

\subsection{Left Poincar\'{e} Map}

%Denoting by $Z^L(x,y)=(Z^L_1(x,y), Z^L_2(x,y))$ the
%left vector field of  Welander system \eqref{eqw3}, then $Z^L_1(0,y)=-y$. Therefore, under the action of $Z^L$, a solution starting from $(0,y_0)$ with $y_0$

According to the last analysis in the Proposition \ref{propcondnec}, the orbits starting from points $(0, y_0)$ with $y_0 > 0$ go into the left zone under the actions of the flow of the left linear system of Welander $Z^{L}$. If these orbits reach $\{x = 0\}$ again at some point $(0, y_1)$ with $y1 < 0$ after some future time, then we can define a Left Poincaré map $P_L$. Without loss of generality, we can define $P_{L}(0)=0$, $P_{L}(y_{0})=y_{1}$, $y_{0}>0$. By \eqref{sollinear} the solution of the left system $Z^{L}$ of Welander in \eqref{eqw3}, with initial condition $(0, y_0)$ with $y_0 > 0$, is given by
\begin{equation*}\label{solZL}
	%\footnotesize 
	\begin{aligned}
		\begin{pmatrix}
			x^{L}(t) \\ 
			\; \\
			y^{L}(t)
		\end{pmatrix}
		&=
		\frac{1}{\lambda_i^{L}-\lambda_j^{L}} \left( \begin{array}{l}
			( e^{\lambda_j^{L}t}- e^{\lambda_i^{L}t} ) \, y_0 + \frac{	\psi_{\lambda_i^{L},\lambda_j^{L}}(t) }{(1+k_0)  (k_0+\beta) } \,  a^{L}\\ 
			\; \\
		(\lambda_i^{L} e^{\lambda_j^{L}t}- \lambda_j^{L} e^{\lambda_i^{L}t}) \, y_0 +	\frac{[(\lambda_i^{L})^2(1-e^{\lambda_j^{L}t})-	(\lambda_j^{L})^2(1-e^{\lambda_i^{L}t})] }{(1+k_0)  (k_0+\beta) } \, a^{L}
		\end{array} \right),
	\end{aligned}
\end{equation*}
where $0>\lambda_{i}^{L}>\lambda_{j}^{L}$, $i,j=1,2$ with $i\neq j$. Taking $x^{L}(t)=0$ solving concerning $y_{0}$ and after replacing $y_{0}^{L}$ in $y^{L}(t)$, we obtain 
%
%By Proposition~\ref{prop1}, the piecewise Welander's system is given by

%\begin{equation}\label{eqwnormal}
%	\begin{aligned}
%		Z(x,y) = \left\{ \begin{array}{lll}
%			Z^{L}(x,y) & = & \left(-\frac{3x}{2} - y,\frac{x}{2} - \frac{1}{10} + \frac{\epsilon}{2}\right),\vspace{0.2cm} \\ 
%			Z^{R}(x,y) & = & \left(-\frac{7x}{2} - y - \epsilon,3x + \frac{1}{5}+ 3\epsilon\right).
%		\end{array} \right.
%	\end{aligned}
%\end{equation}
%
%The eigenvalues are $\lambda_{1}^{L}=-\frac{1}{2}$, $\lambda_{2}^{L}=-1$,  $\lambda_{1}^{R}=-\frac{3}{2}$, $\lambda_{2}^{R}=-2$. Then $\psi^{L}(t)=\frac{1}{2} - e^{-t/2}+ \frac{1}{2} e^{-t}$ and $\psi^{R}(t)=\frac{1}{2} - 2e^{-3t/2}+ \frac{3}{2}e^{-2t}.$
%
%The solution of the left Welander's system, $\dot{\mathrm{x}}=A^{L}\mathrm{x}+b^{L}$, is given by
%
%\begin{equation}\label{sol}
%	\begin{aligned}
%		 \left\{ \begin{array}{lll}
%		x(t) & = & 2(e^{-t} - e^{-t/2})y_{0} + 4(\frac{1}{10} - \frac{\epsilon}{2})(\frac{1}{2}+ \frac{1}{2}e^{-t} - e^{-t/2}),\vspace{0.2cm} \\ 
%			y(t) & = & 2(-\frac{1}{2}e^{-t} + e^{-t/2})y0 + 4(\frac{1}{10}- \frac{\epsilon}{2})(-\frac{3}{4} - \frac{1}{4}e^{-t} + e^{-t/2}).
%		\end{array} \right.
%	\end{aligned}
%\end{equation}

\begin{equation}\label{y0lpm}
y_{0}^{L}=\frac{a^{L} \, \psi_{\lambda_i^{L},\lambda_j^{L}}(t) }{(1+k_0)  (k_0+\beta) ( e^{\lambda_i^{L}t}- e^{\lambda_j^{L}t} )}, \; y_{1}^{L}=\frac{-a^{L} \, e^{-(1+2 \, k_0+\beta)t} \, \psi_{\lambda_i^{L},\lambda_j^{L}}(-t) }{(1+k_0)  (k_0+\beta) ( e^{\lambda_i^{L}t}- e^{\lambda_j^{L}t} )},
\vspace{0.1cm}
\end{equation}
where $ t>0$ and $\psi_{\lambda_i^{L},\lambda_j^{L}}(t)$ is given in \eqref{eqphiL}. Considering $y_{m_1}^{L}$ and $y_{m_2}^{L}$ defined in \eqref{intx=0}, we can describe the next proposition about Left Poincaré Map $P_L$.
%Replacing $y_{0}$ in $y(t)$ we obtain 
%\begin{equation}\label{lpm}
%	y_{1}=\frac{e^{-t/2} (-1 + 5\epsilon)(e^{t/2} - 1)}{10}.
%\end{equation}
%

\begin{proposition}\label{propm}
Considering $\alpha < \alpha^L_{\beta,\varepsilon}$  then Left Poincaré Map $P_L(y_0^L)$ given by \eqref{y0lpm} is well defined and satisfies the following statements.
\begin{enumerate}
	\item $y_{0}^{L}(t)$ is increasing and $y_{1}^{L}(t)$ is decreasing with respect to $t$.
	
	  \vspace{0.1cm}
	  
    \item The domain for $P_{L}$ is $(0,+\infty)$, and
    \begin{enumerate}
    	\item $P_{L}$ is decreasing and convex with respect to $y_{0}^{L}$,
    
      \vspace{0.05cm}
      
    	\item $P_{L}$ has $y_{1}^{L}=y_{m_i}^{L}= a^{L}/\lambda_j^{L}$ as an asymptote.  
    \end{enumerate} 
     
      \vspace{0.1cm}
      
    \item Defining $P_{L}(0)=0$, follow that $P_{L}$ is continuous at $y_{0}^{L}=0$. In addition, follow that
    \begin{enumerate}
   \item $P'_{L}(0)=-1, \;  P''_{L}(0)=\dfrac{4 \, (1+2 \, k_0+\beta)}{3 \, a^{L}},\;  P'''_{L}(0)=-\dfrac{8 \, (1+2 \, k_0+\beta)^2}{3(a^{L})^2},$
      \vspace{0.07cm}
   \item $ P''''_{L}(0)=-\dfrac{16 \, (8 \, (\lambda_i^{L})^3+15 \, (\lambda_i^{L})^2(\lambda_j^{L})+15 \, (\lambda_i^{L})(\lambda_j^{L})^2+8 \, (\lambda_j^{L})^3)}{9 \, (a^{L})^3}. $
\end{enumerate}
\end{enumerate}
\end{proposition}
\begin{proof}
	Deriving $y_{0}^{L}$ and $y_{1}^{L}$ we obtain
\begin{equation}
	\begin{aligned}
	&(y_{0}^{L})^\prime(t)=\frac{a^{L} \, (\lambda_i^{L}-\lambda_j^{L}) \, \psi_{\lambda_i^{L},\lambda_j^{L}}(-t) }{(1+k_0)  (k_0+\beta) ( e^{\lambda_i^{L}t}- e^{\lambda_j^{L}t} )^2}, \\ &(y_{1}^{L})^\prime(t)=\frac{-a^{L} \, (\lambda_i^{L}-\lambda_j^{L}) \, e^{-(1+2 \, k_0+\beta)t} \, \psi_{\lambda_i^{L},\lambda_j^{L}}(t) }{(1+k_0)  (k_0+\beta) ( e^{\lambda_i^{L}t}- e^{\lambda_j^{L}t} )^2},
	\end{aligned}
\end{equation}	
As $\alpha < \alpha^L_{\beta,\varepsilon}$ implies $a^{L}>0$, then by Lemma~\ref{lemma} and Remark~\ref{obsphi}, follow that $y_{0}^{L}$ is increasing and $y_{1}^{L}$ decreasing concerning $t$. In this way, there are uniques $y_{0}^{L}>0$ and $y_{1}^{L}<0$ satisfying  $\eqref{y0lpm}$ for all $t \in (0,+\infty)$. Then $P_{L}: y_{0}^{L} \rightarrow y_{1}^{L}$ is well defined.

By equation \eqref{y0lpm} and Lemma~\ref{lemma}, using the L'hospital rules when necessary, follows that
\begin{equation}
	\begin{array}{lll}
	\displaystyle\lim_{t\rightarrow 0^{+}} y_{0}^{L}(t)=\lim_{t\rightarrow 0^{+}} y_{1}^{L}(t)=0,  & 	\displaystyle\lim_{t\rightarrow +\infty} y_{0}^{L}(t)=+\infty, & 	\displaystyle\lim_{t\rightarrow +\infty} y_{1}^{L}(t)=\frac{a^{L}}{\lambda_{j}^{L}}.
	\end{array}
\end{equation}	
In addition, we have $$P_{L}^\prime(y_{0}^{L})=\dfrac{(y_{1}^{L})^\prime(t)}{(y_{0}^{L})^\prime(t)}=-\dfrac{\psi_{\lambda_i^{L},\lambda_j^{L}}(t) }{\psi_{\lambda_i^{L},\lambda_j^{L}}(-t) }<0,$$ then $P_L$ is decreasing concerning $y_0^{L}$ , its  domain of definition is given by $(0,+\infty)$ and has $y_1^{L}=y_{m_i}^{L}= \frac{a^{L}}{\lambda_j^{L}}$ as an asymptote.
%Then, the domains for $P_{L}$ is $(0,+\infty)$ and the asymptote is $y_{1}=\frac{a^{L}}{\lambda_{2}^{L}}$.
Direct computations shows that $P''_{L}(y_{0}^L)$ is given by
\begin{equation*}\label{dP2}
	\begin{aligned}
	\frac{(y_{0}^{L})^\prime \, (y_{1}^{L})''-(y_{1}^{L})^\prime \, (y_{0}^{L})''}{[(y_{0}^{L})^\prime]^3}=& \frac{[(1+k_0)  (k_0+\beta)]^2 \, e^{-2(1+2k_0+\beta)t} }{-a^L \, (\lambda_i^{L}-\lambda_j^{L})}  \left[ \dfrac{e^{\lambda_i^{L}t}- e^{\lambda_j^{L}t}  }{\psi_{\lambda_i^{L},\lambda_j^{L}}(-t) } \right]^3 \, \sigma(t),
		\end{aligned}
\end{equation*}
with $ \sigma(t)=- (1+2k_0+\beta) \, e^{- (1+2k_0+\beta)t} \, \psi_{\lambda_i^{L},\lambda_j^{L}}(-t) -\psi_{\lambda_i^{L},\lambda_j^{L}}(t)$. The Lemma~\ref{lemma} and Remark~\ref{obsphi} show us that $P''_{L}(y_{0}^L)>0$, that is, $P_{L}$ is convex concerning $y_{0}^L$.

\vspace{0.1cm}

The continuity of $P_{L}$ at  $y_{0}^L=0$ is clear according to the proof of the last statements. Then in the sequence, we will calculate the derivatives at  $y_{0}^L=0$. We can note that the equation of $y_{0}^{L}$ in \eqref{y0lpm} is equivalent to
\begin{equation*}
[(1+k_0)(k_0+\beta) \, y_{0}^{L}-a^{L}\lambda_j^{L}] \, e^{\lambda_i^{L}t}+[a^{L}\lambda_i^{L}-(1+k_0)(k_0+\beta) \, y_{0}^{L}] \, e^{\lambda_j^{L}t}=a^{L}(\lambda_i^{L}-\lambda_i^{L}).
\end{equation*}
%
%Then, the decreasing and convex properties follow from $P'_{L}$ and $P''_{L}$ and the statement (2) is proved. 
Calculating the Taylor expansion concerning $t$ in the last equation, we obtain
\small \begin{equation*}
\begin{aligned}
0&= y_{0}^{L} \, t+\left[ \frac{-(1+2 \, k_0+\beta) \, y_{0}^{L} -a^{L}}{2} \right] \, t^2 + \\
	& \left[ \frac{((\lambda_i^{L})^2+ \lambda_i^{L} \, \lambda_j^{L}+(\lambda_j^{L})^2) \, y_{0}^{L} + (1+2k_0+\beta) \, a^{L}}{6} \right] \, t^3 + \\
	&  \left[ \frac{((\lambda_i^{L})^3+ (\lambda_i^{L} )^2 \lambda_j^{L}+\lambda_i^{L}  (\lambda_j^{L})^2+(\lambda_j^{L})^3) \,  y_0^{L}-((\lambda_i^{L})^2+ \lambda_i^{L} \, \lambda_j^{L}+(\lambda_j^{L})^2) \,  a^{L}}{24} \right] \, t^4 + \mathcal{O}(t^{5}).
	\end{aligned}
\end{equation*}
\normalsize For $y_{0}^{L} >0$ sufficiently small, it is possible to consider $t=a_1 \, y_{0}^{L}+a_2 \, (y_{0}^{L})^2+a_3 \, (y_{0}^{L})^3+a_4 \, (y_{0}^{L})^4+ \mathcal{O}((y_{0}^{L})^{5})$, and invert the last series (see \cite{Ponce}). Then replacing the new series of $t$ in the las equation and solving concerning $a_i$, we obtain
\begin{equation}\label{seriest1}
\begin{aligned}
t=
& \dfrac{2}{a} \, y_{0}^{L}- \dfrac{2 \, (1+2\, k_0 +\beta)}{3 \, a^{L}} \, (y_{0}^{L})^2+ \dfrac{2}{9} \, \dfrac{(2 \, (\lambda_i^{L})^2+ \lambda_i^{L} \, \lambda_j^{L}+2 \, (\lambda_j^{L})^2) }{(a^{L})^3} \, (y_{0}^{L})^3 +\\
& \dfrac{4}{27} \, \dfrac{(4 \, (\lambda_i^{L})^3+ 3 \, (\lambda_i^{L} )^2 \lambda_j^{L}+ 3 \, \lambda_i^{L}  (\lambda_j^{L})^2+ 4 \, (\lambda_j^{L})^3) }{(a^{L})^4}  \, (y_{0}^{L})^4 + \mathcal{O}(t^{5}).
\end{aligned}
\end{equation}
%Calculating the Taylor expansion concerning $t$ in the last equation, we obtain
%Let us called $y_{0}=\delta$ in $x(t)$ and let us do a Taylor expansion in $x(t)$ concerning to $t$. Then we have
%\begin{equation}\label{seriest}
%-t\delta - \dfrac{(-15\delta - 1 + 5\epsilon)t^2}{20} + \dfrac{(-35\delta - 3 + 15\epsilon)t^3}{120} - \dfrac{(-75\delta - 7 + 35\epsilon)t^4}{960}+O(t^{5}).
%\end{equation}
%
%Taking $t_{1}=a_{1}\delta+a_{2}\delta^2+a_{3}\delta^3+a_{4}\delta^4+$, and replacing in \eqref{seriest}, and solving concerning $a_{i}, i=1\ldots4$, we obtain
%\begin{equation}\label{seriest1}
%t_{1}=\frac{-40000\delta^5}{(-1 + 5\epsilon)^5} - \frac{5000\delta^4}{(-1 + 5\epsilon)^4} - \frac{2000\delta^3}{(3(-1 + 5\epsilon)^3)} - \frac{100\delta^2}{(-1 + 5\epsilon)^2} - \frac{20\delta}{(-1 + 5\epsilon)}.
%\end{equation}

The next step consist in doing a Taylor expansion for  $y(t)^{L}$ concerning to $t$, and replacing $t$ by \eqref{seriest1}. In this way, after some rearrangements we obtain
\begin{equation}
	\begin{aligned}
P_{L}(y_{0}^{L}) := & -y_{0}^{L} +  \dfrac{2 \, (1+2\, k_0 +\beta)}{3 \, a^{L}} \, (y_{0}^{L})^2 -  \dfrac{4 \, (1+2\, k_0 +\beta)^2}{9 \, (a^{L})^2} \, (y_{0}^{L})^3 - \\
& \dfrac{2 \, (8 \, (\lambda_i^{L})^3+ 15 \, (\lambda_i^{L} )^2 \lambda_j^{L}+ 15 \, \lambda_i^{L}  (\lambda_j^{L})^2+ 8 \, (\lambda_j^{L})^3) }{27 \, (a^{L})^3}  \, (y_{0}^{L})^4 + \mathcal{O}(t^{5}).
\end{aligned}
\end{equation}
Deriving $P_{L}({y_{0}^{L}})$ concerning $y_{0}^{L}$ and taking $y_{0}^{L}=0$, the statement (3) follows.
\end{proof}

\subsection{Right Poincar\'{e} Map}
As observed in the Proposition \ref{propcondnec}, the orbits starting from points $(0, y_0)$ with $y_0 < (k_0-k_1) \, \varepsilon$ go into the right zone under the actions of the flow of the right linear system of Welander $Z^{R}$. If these orbits reach $\{x = 0\}$ again at some point $(0, y_1)$ with $y_1 > (k_0-k_1) \, \varepsilon$ after some future time, then we can define a Right Poincaré map $P_R$. Without loss of generality, we can define $P_{R}((k_0-k_1) \, \varepsilon)=(k_0-k_1) \, \varepsilon$, $P_{R}(y_{0})=y_{1}$, $y_{0}<(k_0-k_1) \, \varepsilon$. Using \eqref{sollinear} the solution of the right system $Z^{R}$ of Welander in \eqref{eqw3}, with initial condition $(0, y_0)$ with $y_{0}<(k_0-k_1) \, \varepsilon$, and following as the left Poincaré Map we have
\begin{equation}\label{y0rpm}
	\begin{aligned}
	& y_{0}^{R}=\frac{a^{R} \, \psi_{\lambda_i^{R},\lambda_j^{R}}(s) }{(1+k_1)  (k_1+\beta) ( e^{\lambda_i^{R}s}- e^{\lambda_j^{R}s} )} + (k_0-k_1)\varepsilon, \\
	\vspace{0.2cm}
	& y_{1}^{R}=\frac{-a^{R} \, e^{-(1+2 \, k_1+\beta)s} \, \psi_{\lambda_i^{R},\lambda_j^{R}}(-s) }{(1+k_1)  (k_1+\beta) ( e^{\lambda_i^{R}s}- e^{\lambda_j^{R}s} )} +(k_0-k_1)\varepsilon,
	\vspace{0.1cm}
	\end{aligned}
\end{equation}
where $ s>0$ and $\psi_{\lambda_i^{R},\lambda_j^{R}}(s)$ is given in \eqref{eqphiL}. Therefore, for $y_{m_1}^{R}$ and $y_{m_2}^{R}$ give in \eqref{intx=0}, we have the next proposition about Right Poincaré Map $P_R$ such that the proof is similar to Proposition~\ref{propm}.

\begin{proposition}\label{propp}
	Considering $\alpha > \alpha^R_{\beta,\varepsilon}$  and $B=(k_0-k_1) \, \varepsilon$ then Right Poincaré Map $P_R(y_0^R)$ given by \eqref{y0rpm} is well defined and satisfies the following statements.
	\begin{enumerate}
		\item $y_{0}^{R}(s)$ is decreasing  and $y_{1}^{R}(s)$ is increasing concerning to $s$.
		
		\vspace{0.1cm}
		
		\item The domain for $P_{R}$ is $(-\infty,B)$, and
		\begin{enumerate}
			\item $P_{R}$ is decreasing and concave concerning to $y_{0}^{R}$,
			
			\vspace{0.05cm}
			
			\item $P_{R}$ has $y_{1}^{R}=y_{m_i}^{R}+B= a^{R}/\lambda_j^{R}+B$ as an asymptote.  
		\end{enumerate} 
		
		\vspace{0.1cm}
		
		\item Defining $P_{R}(B)=B$, follow that $P_{R}$ is continuous at $y_{0}^{R}=B$. Moreover, follow that
		\begin{enumerate}
			\item $P'_{R}(B)=-1, \,  P''_{R}(B)=\dfrac{4 \, (1+2 \, k_1+\beta)}{3 \, a^{R}}, \,  P'''_{L}(B)=-\dfrac{8(1+2 \, k_1+\beta)^2}{3(a^{R})^2},$
			\vspace{0.07cm}
			\item $ P''''_{R}(B)=-\dfrac{16 \, (8 \, (\lambda_i^{R})^3+15 \, (\lambda_i^{R})^2(\lambda_j^{R})+15 \, (\lambda_i^{R})(\lambda_j^{R})^2+8 \, (\lambda_j^{R})^3)}{9 \, (a^{R})^3}. $
		\end{enumerate}
	\end{enumerate}
\end{proposition}
%\begin{proof}
%	The proof is similar to Proposition~\ref{propm}.
%\end{proof}

\subsection{The Full  Poincar\'{e} Map}
Now we are ready to define the Full Poincar\'{e} Map, $P=P_{R}\circ P_{L}$. The proof follows from using the chain rule in $P$ and from Propositions~\ref{propm} and \ref{propp}.

\begin{proposition}\label{proptotal}
		Let us consider $\alpha < \alpha^L_{\beta,\varepsilon}$ and  $\alpha > \alpha^R_{\beta,\varepsilon}$ with $\alpha^{L,R}_{\beta,\varepsilon}$ given in \eqref{alphaLR}. If $B=(k_0-k_1)\varepsilon=0$ in the system of Welander \eqref{eqw3}, then the Full Poincaré Map is well defined and satisfies
\begin{equation*} 
	P'(0)=-1, \; \; P''(0)=\dfrac{4}{3} \, \left[  \dfrac{(1+2 \, k_1+\beta)}{a^R} -\dfrac{(1+2 \, k_0+\beta)}{a^L}\right], \; \; P'''(0)=\dfrac{3}{2} \, \left[P'(0)\right].
\end{equation*}
Moreover, if $(1+2 \, k_1+\beta)/a^R=(1+2 \, k_0+\beta)/a^L $ then
\begin{equation*}
	\begin{aligned}
	 & \; \; P'(0)=-1, \qquad P''(0)= P'''(0)=0, \\
	 \vspace{0.1cm}
	 & P''''(0)=\dfrac{16 \, (1+2 \, k_1+\beta)}{a^R} \, \left[\dfrac{(1+k_0) (k_0+\beta)}{(a^L)^2}-\dfrac{(1+k_1) (k_1+\beta)}{(a^R)^2}\right].
	 \end{aligned}
\end{equation*}
\end{proposition}

\begin{remark}\label{remare}
	Observe that $k_0-k_1<0$, then we have $B=(k_0-k_1)\varepsilon=0$ $(>0,<0)$ if , and only if, $\varepsilon=0$ $(<0,>0)$ . By Proposition \ref{propeqdob}, $B$ coincides with the second coordinate $(y_t^R)$ of the fold point of the $Z^R$.
	
\end{remark}
%\begin{proof}
%	The proof follows from using the chain rule in $P=P_{R}\circ P_{L}$ and from Prepositions~\ref{propm} and \ref{propp}.
%\end{proof}

\section{Limit cycle in Welander's system}\label{LMW}

In this section, we will study the existence of limit cycles for Weelander's system. Firstly, we will give another necessary condition to have periodic orbits. We observe that in the canonical form given by Proposition \ref{prop1},
%is invariant under the change of coordinates $(x,y,t) \rightarrow (x,-y,-t)$ simultaneous with the change of parameters
%\begin{equation}\label{change}
%\begin{aligned}	
%	& (det(A^{R}),tr(A^{R}),a^{R},b) \rightarrow (det(A^{R}),-tr(A^{R}),a^{R},-b); \\
%& \phantom{,B} \, (det(A^{L}),tr(A^{L}),a^{L}) \rightarrow (det(A^{L}),-tr(A^{L}),a^{L}).
%\end{aligned}
%\end{equation}
%Therefore, it is enough to consider  $b\geq 0$ in the study of the system in canonical form \eqref{eqwnormal3} and \eqref{eqwnormal4}. Then, s
when $y>b > 0$ we have the first coordinate of the respective vector fields always negative on the discontinuity line and, if $y<0< b$ we have it always positive which permits a sliding segment for $0 < y < b$. The case $y<b<0$ and $y>0>b$ generates opposite signs in the previous analysis with an escaping segment in $b<y<0$.

In our situation, a crossing periodic $\Gamma$ orbit is formed by an arc of the left vector field $\Gamma^L$ and another arc of the right vector field $\Gamma^R$, that is, $\Gamma = \Gamma^L \cup \Gamma^R.$ In addition, the periodic orbit $\Gamma$ has exactly two points at the $y-axis$. We will call these two points as $(0,y_{L})$ and $(0,y_{R})$, and with the presented configuration we have $y_L<y_R$, that is, $y_R=y_L +h$ with $h>0$. Then, to enunciate the next proposition that will give more than one necessary condition to have a periodic orbit, we will consider some new notations. Since $\Gamma$ is a closed Jordan curve, it is possible to consider $\Omega^{L}=int\{\Gamma^{L} \cup F^{L}\}$, $\Omega^{R}=int\{\Gamma^{R}\cup F^{R}\}$, $\sigma^{L}=area(\Omega^{L})$, and $\sigma^{R}=area(\Omega^{R})$  such that $F^{L}$ and $F^{R}$ consist to the following oriented segments
\begin{equation*}
	\begin{aligned}
 	&F^{L}=\{x=0, \, y=(1-\mu) y^{L} + \mu y^{R}, \; 0 \leq \mu \leq 1\}, \\
 	& F^{R}=\{x=0, \, y= \mu y^{L} + (1-\mu) y^{R}, \; 0 \leq \mu \leq 1\}.
 	\end{aligned}
\end{equation*}
The Figure \ref{prop} represent the discussion and notation describes until here.
%Since $y>b$, we have that $Z^{R}(0,y)<0$ and $Z^{L}(0,y)<0$. For $y<0$, we have that $Z^{R}(0,y)>0$ and $Z^{L}(0,y)>0$. 
%Then a crossing periodic orbit $\Gamma=\Gamma^{L}\bigcup\Gamma^{R}$ has exactly two points at the $y-axis$. We will call these two points $(0,y_{L})$ and $(0,y_{R})$. Let us consider $\Omega^{L}=int{\Gamma^{L}}$, $\Omega^{R}=int{\Gamma^{R}}$, $\sigma^{L}=area(\Omega^{L})$, and $\sigma^{R}=area(\Omega^{R})$.
\begin{center}
	\begin{figure}[H]
		\begin{overpic}[scale=0.5]
			{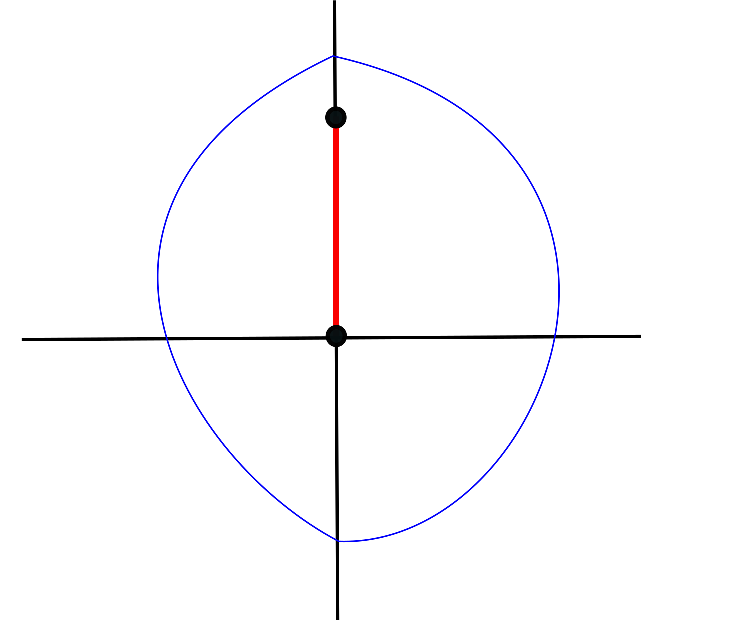} 
			\put(25,70){$\Gamma^{L}$}
			\put(47,80){\textbf{$y_{R}$}}
			\put(47,5){\textbf{$y_{L}$}}
			\put(70,65){$\Gamma^{R}$}
			\put(35,50){$\Omega^{L}$}
			\put(50,50){$\Omega^{R}$}
		\end{overpic}
		\vspace{0.4cm}
		\caption{Crossing periodic orbit surrounding the sliding set. }\label{prop}
	\end{figure}
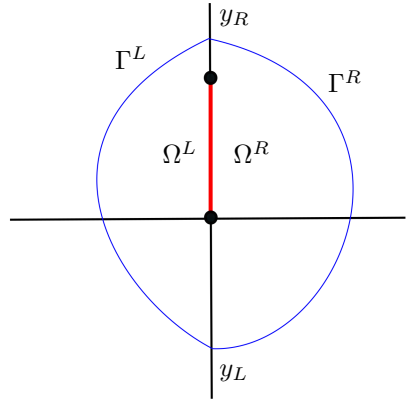
\end{center}

\vspace{-1.3cm}

\begin{proposition}[\cite{Ponce}]\label{propponce}
	Considering piecewise system given by Proposition \ref{prop1} and suppose that there exist a crossing periodic orbit $\Gamma$ that crosses the line $x=0$ through the points $(0,y_{L})$ and $(0,y_{L}+h)$, where $h>0$, then 
	\begin{equation}\label{eqtrarea}
		tr(A^{L}) \, \sigma^{L}+tr(A^{R}) \, \sigma^{R}+b \, h=0.
	\end{equation} 
\end{proposition}

The result above is also a necessary condition for the existence of periodic orbits. In addition, we can use the Proposition~\ref{propponce} to show that there are no crossing periodic orbits. When $tr(A^{L}) \, tr(A^{R}) \geq 0$ such that $tr(A^{L})+tr(A^{R})\neq0$, and $b=0$, the equality \eqref{eqtrarea} can not be fulfilled, and then, we can conclude that there are no periodic orbits.

\begin{proposition}\label{main}
	 If $\alpha \, (1-\beta) \neq 0$ such that $\alpha < \alpha^L_{\beta,\varepsilon}$ and  $\alpha > \alpha^R_{\beta,\varepsilon}$ with $\alpha^{L,R}_{\beta,\varepsilon}$ given in \eqref{alphaLR}, then for the piecewise smooth Weelander's system given by \eqref{eqw3}, we have the following statements:
	\begin{enumerate}
		\item The piecewise Welander's system does not have crossing periodic orbits when there is or not sliding segment.
		\vspace{0.1cm}
	\item The piecewise Welander's system has a unique and stable crossing periodic orbit surrounding the escaping segment.
	\end{enumerate}
\end{proposition}

\begin{proof} 

Note that simple calculations between the relation of the derivative of a function and its inverse and the Proposition \ref{propp}, show us that the first and second derivatives of $P_R^{-1}$ have the same sign as the first and the second derivative of $P_R$, that is, $P_R^{-1}$ is also decreasing and concave. Furthermore, in the piecewise smooth Welander system, we have the parameter $b=B=(k_0-k_1) \, \varepsilon$, where $b$ is highlighted in the Proposition~\ref{propponce}.  

In order to proof the first part of the statement (1), the equation \eqref{eqlambda} implies  $$tr(A^{L})=-(1+2 \, k_0+\beta)<0, \qquad tr(A^{R})=-(1+2 \, k_1+\beta)<0,$$
consequently $tr(A^{L}) \, tr(A^{R}) >0$ and $tr(A^{L})+tr(A^{L})\neq0$. The Proposition~\ref{prop7} guarantees that $\varepsilon=0$ ($\varepsilon>0$) is equivalent to the non-existence (existence) of sliding segment and, using the Remark~\ref{remare} when $\varepsilon$ vanishes also vanishes the parameter $b=B$, by the Proposition \ref{propponce} there are no crossing periodic orbits when we do not have sliding segment. Now for $\varepsilon>0$, using the Proposition \ref{propm} we have $P_L$ decreasing and convex in its domain, which guarantees that the graph of $P_L$ is above its tangent line at $y_0=0$, $y_1=-y_0.$ In the same way, the Proposition \ref{prop} guarantees that the graph $P_R^{-1}$ is above its tangent line at $y_0=B$, $y_1=-y_0 + 2 \, B.$ Consequently, the graphs of $P_L$ and $P_R^{-1}$ never intersect, see Figure~\ref{PLPR}, then there are also not crossing period orbits, and the statement (1) is proved.

% As $b=B=0$, follow that there are no crossing periodic orbits when we do not have sliding segment.

%We have that $tr^{L}=-\dfrac{3}{2}$, $tr^{R}=-\dfrac{7}{2}$. Then, we have $tr^{L}tr^{R}\geq0$, $tr^{L}+tr^{R}\neq0$. 
%
%As $b\neq0$, follow that there are no periodic orbits and statement (1) is proved.                           
	
%To prove statement (2), we observe that using Proposition~\ref{propm} follow that the graph of $P_{L}$ is below the line $y_{1}=-y_{0}$. Now, using Proposition~\ref{propp}, follow that the graph of $P_{R}^{-1}$ is above the line $y_{1}=-y_{0}+2b$. Then, the graph of $P_{L}$ is below the graph of $P_{R}^{-1}$ and there are no periodic orbits. 
As Proposition~\ref{prop7} guarantees that $\varepsilon<0$ is equivalent to the existence of escaping segment, then to prove statement (2), we have that Proposition~\ref{propm} guarantees  $P_L$  decreasing and convex in its domain $[0,+\infty)$ then its derivative is increasing, and the Proposition ~\ref{propp}  shows $P_R^{-1}$ is decreasing and concave in its domain $[B,a^{r}/\lambda_{j}^{R}+B)$ then we have its derivative decreasing. Therefore for $y_0>B>0$ follows that
\begin{equation}\label{crescentedisplayfunc}
	\begin{aligned}
		& P_L(y_0)<P_L(B)<P_L(0)=0, \\
 &P_L'(y_0) >P_L'(B)>P_L'(0)=-1=(P_{R}^{-1})'(B)>(P_{R}^{-1})'(y_{0}).
 \end{aligned}
\end{equation}
Let us define now the displacement function $\Delta(y_{0})=P_{L}(y_{0})-P_{R}^{-1}(y_{0}),$ such that by equation \eqref{crescentedisplayfunc} we have
$\Delta(B)=P_L(B)-P^{-1}_{R}(B)<-B<0$ and $\Delta'(y_{0})=P'_{L}(y_{0})-(P_{R}^{-1})'(y_{0})>0$. In addition, 
\begin{equation*}
	\vspace{-0.2cm}
	\lim_{y_{0}\rightarrow B+a^{r}/\lambda_{j}^{R}}\Delta(y_{0})=P_L(B+a^{r}/\lambda_{j}^{R})-(-\infty)>0.
		\vspace{0.15cm}
\end{equation*}
Therefore, as $\Delta(y_{0})$ is increasing concerning $y_{0}$, $\Delta(B)<0$ and for $B<y_0$ sufficiently near of $B+a^{r}/\lambda_{j}^{R}$, we have $\Delta(y_0)>0$ follows that there is a unique $\bar{y_{0}}\in (B,a^{r}/\lambda_{j}^{R}+B)$ such that $\Delta(\bar{y_{0}})=0$, that is, there is a unique crossing periodic orbit. Moreover, the crossing periodic orbit is stable since $P_{L}(y_{0})<P_{R}^{-1}(y_{0}) \, (\Delta(y_0)<0)$ for $B \leq y_0 <\bar{y_{0}}$ and $P_{L}(y_{0})>P_{R}^{-1}(y_{0}) \, (\Delta(y_0)>0)$ for $\bar{y_{0}}<y_0<B+a^{r}/\lambda_{j}^{R}$, that is, the orbits near to the periodic orbit are converging to it.

%On the other hand, by statement (2) of Proposition~\ref{propm} we have that $P_{L}$ is decreasing and convex with domain $[0,+\infty)$. In the other hand, $P_{R}^{-1}$ is decreasing and concave with domain $[b,b+a^{r}/\lambda_{2}^{R})$. 

%Let us define now the displacement function $\Delta(y_{0})=P_{L}(y_{0})-P_{R}^{-1}(y_{0}).$ Then, using statement (3) of Propositions~\ref{propm} and \ref{propp}, follow that  $H'(y_{0})=P'_{L}(y_{0})-(P_{R}^{-1})'(y_{0})>0$. Therefore, $H(y_{0})$ is increasing concerning $y_{0}$.

%
% Moreover, $H(\epsilon)<0$ and $\displaystyle\lim_{y_{0}\rightarrow b+a^{r}/\lambda_{2}^{R}}H(y_{0})>0$. Then, there is a unique $\bar{y_{0}}\in (b,b+a^{r}/\lambda_{2}^{R})$ such that $H(\bar{y_{0}})=0$. Thus, there are a unique periodic orbit. 
 
\end{proof}

Therefore, using the Propositions~\ref{prop3}, \ref{prop11} and \ref{main}, the Theorems~\ref{noperio} and \ref{limit} are proved. 
\begin{center}
	\begin{figure}[H]
		\begin{overpic}[scale=0.48]
			{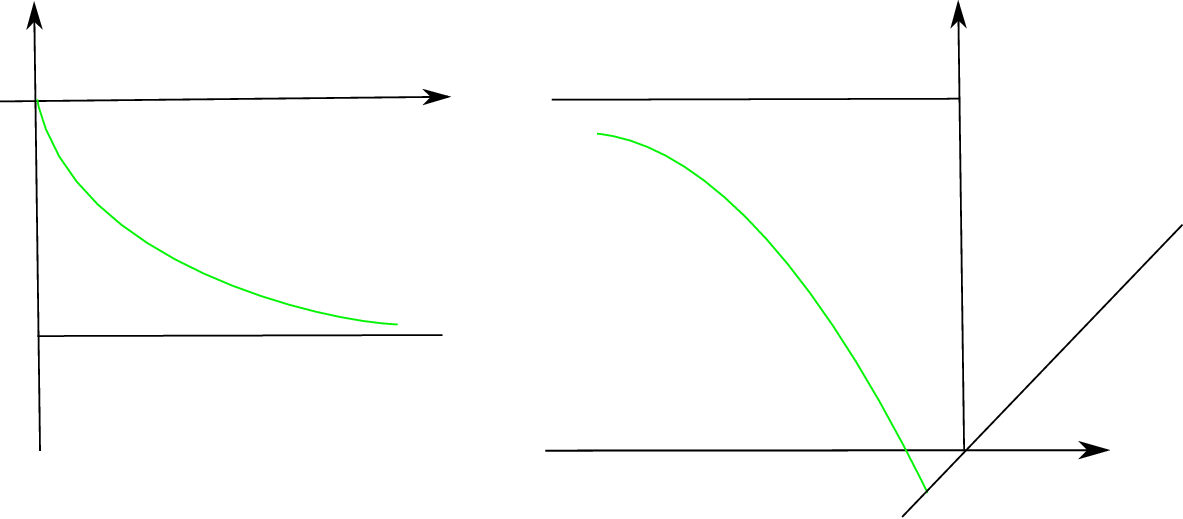} 
			\put(-2,45){$y_{1}$}
			\put(40,35){$y_{0}$}
			\put(-5,15){$y^{L}_{m_{i}}$}
			\put(77,45){$y_{1}$}
			\put(95,5){$y_{0}$}
			\put(82,35){$b+y^{R}_{m_{i}}$}
			\put(95,26){$y_{1}=y_{0}$}
		\end{overpic}
		\vspace{0.4cm}
		\caption{Graphs of $P_{L}$ and $P_{R}$. }\label{PLPR}
	\end{figure}
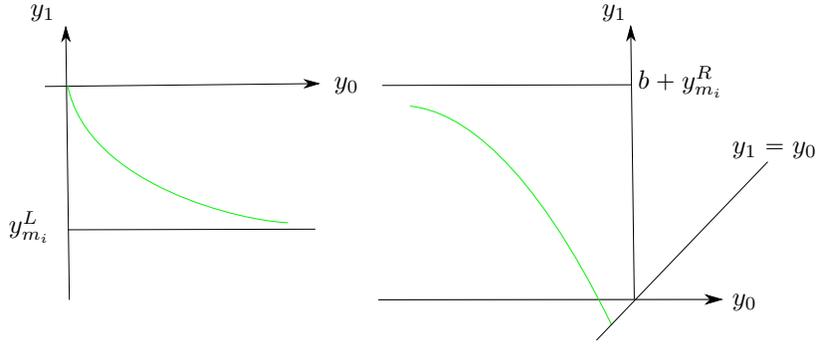
\end{center}

\vspace{-1.7cm}

\begin{remark}
We would like to obverse that we adapt the Propositions~\ref{propm}, \ref{propp}, \ref{proptotal}, and \ref{main} to Weelander's model. As we mentioned, this work followed the theory developed in \cite{node}, where the reader can find the original versions of the propositions.
\end{remark}

\subsection{Example of application} 
Welander \cite{welander} analyzed its model \eqref{wel}, using the smooth function \eqref{ksuave} and the parameters
	\vspace{-0.2cm}
\begin{equation*}\label{paramwelandercont}
	\beta=1/2, \qquad \alpha = 4/5,  \qquad a=1/500, \qquad \varepsilon=-1/30.
		\vspace{-0.2cm}
\end{equation*}
He showed that the respective system obtained has a unique unstable equilibrium point for $(T,S)=(2/3,1,2)$ with $\rho=-1/30$
%\begin{equation}\label{eqwelandercont}
%	T=\dfrac{2}{3}, \qquad S=\dfrac{1}{2}, \qquad \rho=-\dfrac{1}{30},
%\end{equation}
and, calculating by numerical integration verified the existence of a stable periodic orbit or a stable limit cycle in this case, see Figure \ref{fig1}. Although the previous analysis does not can apply in the case of the non-smooth system with $k(\rho)$ given in \eqref{funck}, Welander still numerically finds a periodic orbit using the parameters
%$$\beta=\dfrac{1}{10}, \qquad \alpha = \dfrac{1}{5},  \qquad k_1=5, \qquad k_0=0, \qquad \varepsilon=-\dfrac{1}{100}.$$
\begin{equation*}\label{paramwelandercont}
	\beta=1/10, \qquad \alpha = 1/5,  \qquad a=1/500, \qquad k_1=5, \qquad k_0=0, \qquad  \varepsilon=-1/100.
	\vspace{-0.2cm}
\end{equation*}

\begin{center}
\vspace{-0.15cm}
	\begin{figure}[H]
		\begin{overpic}[scale=0.47]
			{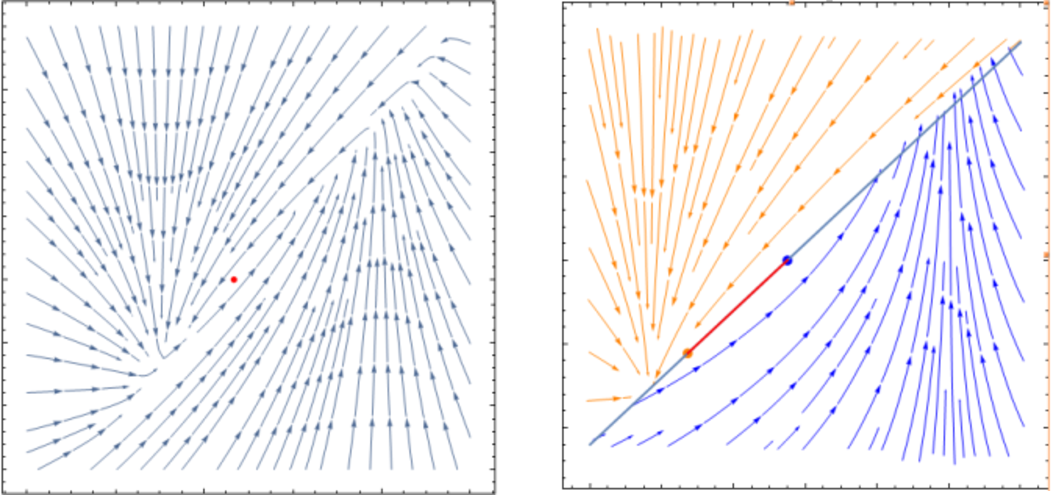} 
				\put(21,-7){\textbf{(a)}}
			\put(73,-7){\textbf{(b)}}
		\end{overpic}
		\vspace{0.8cm}
		\caption{Phase portrait of the continuous (a) and piecewise (b) Welander's system.}\label{fig1}
	\end{figure}
\end{center}

\vspace{-0.1cm}

%\begin{center}
%	\begin{figure}[H]
%		\begin{overpic}[scale=0.28]
%		{retratofase_cont.png} 
%%			\put(7,91){$T_A$}
%%			\put(50,91){$S_A$}
%%			\put(13,74){$k_{T}$}
%%			\put(56,74){$k_{S}$}
%%			\put(29,54){$T,S$}
%%			\put(41,39.5){$k(\rho)$}
%%			\put(19,22){$T_{0}=S_{0}=0$}
%		\end{overpic}
%		\vspace{0.5cm}
%		\caption{Phase portrait of the continuous Welander system \eqref{wel}, considering \eqref{ksuave}, \eqref{paramwelandercont} and \eqref{eqwelandercont}.}\label{retratofase_cont}
%	\end{figure}
%\end{center}

\vspace{-1cm }

%\begin{center}
%	\begin{figure}[H]
%		\begin{overpic}[scale=0.25]
%			{retratofase_descont.png} 
%			%			\put(7,91){$T_A$}
%			%			\put(50,91){$S_A$}
%			%			\put(13,74){$k_{T}$}
%			%			\put(56,74){$k_{S}$}
%			%			\put(29,54){$T,S$}
%			%			\put(41,39.5){$k(\rho)$}
%			%			\put(19,22){$T_{0}=S_{0}=0$}
%		\end{overpic}
%		\vspace{0.5cm}
%		\caption{Phase portrait of the continuous Welander system \eqref{wel}, considering \eqref{funck} and \eqref{paramwelandercont}.}\label{retratofase_cont}
%	\end{figure}
%\end{center}

Clearly, Welander chooses the respective non-smooth version because it is a simpler model from some perspective, for example, the non-smooth system has linear equations. It is not obvious that the smooth and the non-smooth systems have the same qualitative behavior in all parameter ranges. Moreover, Welander used different parameter values to exemplify the same type of oscillation. Therefore in the next example, we will choose the same parameters between the smooth and non-smooth systems such that the non-smooth system is the point-wise limit of the smooth system as $a \rightarrow 0$, and then the systems can be compared, with $\varepsilon$ being the bifurcation parameter. In this way, we will choose the following parameters
%$$\beta=\dfrac{1}{2}, \qquad \alpha = \dfrac{4}{5},  \qquad k_1=1, \qquad k_0=0.$$
\begin{equation*}\label{paramwelandercont}
	\beta=1/2, \qquad \alpha = 4/5,  \qquad k_1=1, \qquad k_0=0.
\end{equation*}
With this, after the change of coordinates \eqref{mudcoord} in the Welander's system \eqref{wel}, we obtain a piecewise smooth system in which it is possible to apply the change of coordinates in the Proposition \ref{prop1}. Therefore, in the notation of Theorem \ref{limit}  we consider the following piecewise smooth system
\begin{equation}\label{eqw3subs}
	\begin{aligned}
		&  Z^{L}(x,y)  = \left(
		\begin{array}{rr}
			-3/2 & -1  \\
		1/2 &  0
		\end{array} \right)
		\left(
		\begin{array}{r}
			x \\
			y
		\end{array} \right)
		-
		\left(
		\begin{array}{c}
			0 \\
			1/10-\varepsilon/2
		\end{array} \right)  \; \mbox{if}, \; x<0,\\
		\; \\
		&  Z^{R}(x,y) = \left(
		\begin{array}{rr}
			-7/2 & -1  \\
			3  &  0
		\end{array} \right)
		\left(
		\begin{array}{r}
			x \\
			y
		\end{array} \right)
		-
		\left(
		\begin{array}{c}
			-(-1) \, \varepsilon  \\
		-	(1/5+3\, \varepsilon)
		\end{array} \right) \, \mbox{if} \,x>0, \\
	\end{aligned}
\end{equation}
with the eigenvalues $\lambda_{1}^{L}=-1/2$, $\lambda_{2}^{L}=-1$, and $\lambda_{1}^{R}=-3/ 2$, $\lambda_{2}^{R}=-2$. The system \eqref{eqw3subs} have two virtual attractors node, $(1/5-\varepsilon,-3/10+3 \,  \varepsilon/2)$ for $Z^{L}$, and $(-1/15-\varepsilon,7/30 \, \varepsilon/2)$ for $Z^{R}$, also have two invisible folds, $(0,0)$ for $Z^{L}$ and, $(0,-\varepsilon)$ for $Z^{R}$. Moreover, $\alpha^L_{\beta,\varepsilon}=1-\varepsilon$, $\alpha^R_{\beta,\varepsilon}=2/3-2 \, \varepsilon$, and clearly, 
%Then $\psi^{L}(t)=\frac{1}{2} - e^{-t/2}+ e^{-t/2}$ and $\psi^{R}(t)=\frac{1}{2} - 2e^{-3t/2}+ \frac{3}{2}e^{-2t}.$

%The solution of the left Welander's system, $\dot{\mathrm{x}}=A^{L}\mathrm{x}+b^{L}$, is given by
%
%\begin{equation}\label{sol}
%	\begin{aligned}
%		\left\{ \begin{array}{lll}
%			x(t) & = & 2(e^{-t} - e^{-t/2})y_{0} + 4(\frac{1}{10} - \frac{\epsilon}{2})(\frac{1}{2}- e^{-t/2} + e^{-t/2}),\vspace{0.2cm} \\ 
%			y(t) & = & 2(-e^{-t/2} + e^{-t/2})y0 + 4(\frac{1}{10}- \frac{\epsilon}{2})(-\frac{3}{4} - e^{-t/4} + e^{-t/2}).
%		\end{array} \right.
%	\end{aligned}
%\end{equation}
%
%Taking $x(t)=0$ and solving concerning $y_{0}$, we obtain 
%
%
%The solution of the left Welander's system, $\dot{\mathrm{x}}=A^{R}\mathrm{x}+b^{R}$, is given by
%
%\begin{equation}\label{solr}
%	\begin{aligned}
%		\left\{ \begin{array}{lll}
%			x(t) & = & \dfrac{2(-30y_{0} + 90\epsilon + 4)e^{-3t/2}}{15} + \dfrac{(30y_{0} - 75\epsilon - 3)e^{-2t}}{15} - \epsilon - \dfrac{1}{15},\vspace{0.2cm} \\ 
%			y(t) & = & \dfrac{(120y_{0} - 360\epsilon - 16)e^{-3t/2}}{30} + \dfrac{(-90y_{0} + 225\epsilon + 9)e^{-2t}}{30} + \dfrac{9\epsilon}{2} + \dfrac{7}{30}.
%		\end{array} \right.
%	\end{aligned}
%\end{equation}
%
%Taking $x(t)=0$ and solving concerning $y_{0}$, we obtain 
\begin{equation*}\label{y0lpmy0rpmsubs*}
	\begin{aligned}
		&	y_{0}^{L}=-\frac{(-1 + 5\epsilon)(e^{t/2} - 1)}{10},\quad y_{0}^{R}=\frac{(-1 + e^{s/2} )  (3+2 \, e^{s/2} + e^{s})  (1+15 \, \varepsilon )}{30} -\varepsilon, \\
		\vspace{0.15cm}
		& y_{1}^{L}=\frac{(-1 + 5\epsilon)(e^{-t/2} - 1)}{10},  y_{1}^{R}=\frac{(-1 + e^{s/2} )  (1+2 \, e^{s/2} + 3 \, e^{s})  (1+15 \, \varepsilon ) \, e^{-3s/2} }{30} -\varepsilon.
%		\vspace{0.1cm}
	\end{aligned}
\end{equation*}
Then when $\varepsilon \geq 0$ the system does not have crossing periodic orbits and when $\varepsilon < 0$ the system has a unique stable crossing periodic orbit, see Figure~\ref{fig2}. A rigorous treatment of Welander’s Non-smooth Model \eqref{eqw3subs} can be seen in \cite{leifeld2016smooth}, where the authors compare the smooth and non-smooth models in this situation of parameters.
\begin{center}
	\begin{figure}[H]
		\begin{overpic}[scale=0.391]
			{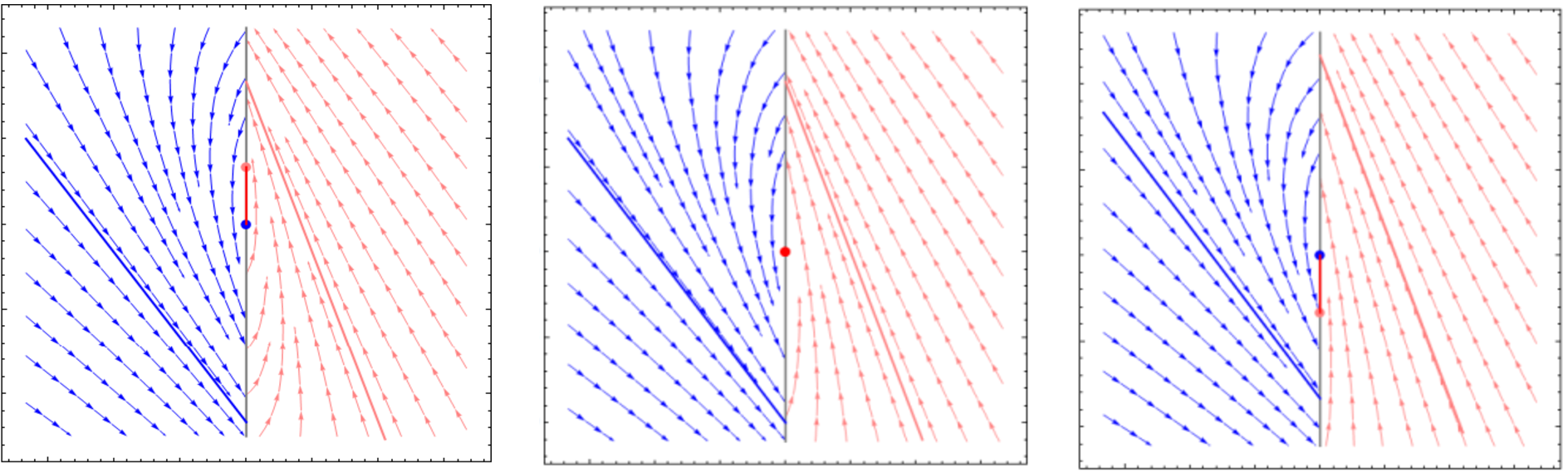} 
			\put(13,-5){\textbf{(a)}}
			\put(49,-5){\textbf{(b)}}
			\put(82,-5){\textbf{(c)}}
		\end{overpic}
		\vspace{1cm}
		\caption{Phase portrait of the non-smooth Welander's system. In picture \textbf{(a)} we have $\varepsilon<0$, in picture \textbf{(b)} we have $\varepsilon=0$, and in picture \textbf{(c)} we have $\varepsilon>0$  }\label{fig2}
	\end{figure}
\end{center}

	\vspace{-1cm}

\section*{Acknowledgements}
%The authors would like to thank Prof. Dr. Douglas Novaes for the discussions that helped us to develop this work. We also would like to thank the staff of the Mathematics Department at The University of Minnesota to welcome us very well. 
The authors would like to thank Prof. Dr. Douglas Novaes for the discussions that helped us to develop this work. We also would like to thank the staff of the Mathematics Department at The University of Minnesota for all the support.

Yagor Romano Carvalho was supported by São Paulo Paulo Research Foundation (FAPESP) grants number 2022/03800-7, 2021/14695-7.  Luiz F.S. Gouveia was supported by São Paulo Paulo Research Foundation (FAPESP) grants number 2022/03801-3, 2020/04717-0.

\bibliographystyle{abbrv}
\bibliography{references1}

\begin{thebibliography}{10}

\bibitem{BerBudCham2008}
M.~Bernardo, C.~Budd, A.~R. Champneys, and P.~Kowalczyk.
\newblock {\em Piecewise-smooth dynamical systems: theory and applications},
  volume 163.
\newblock Springer Science \& Business Media, 2008.

\bibitem{SDB}
H.~W. Broer, F.~Dumortier, S.~J. van Strien, and F.~Takens.
\newblock {\em Structures in dynamics}, volume~2 of {\em Studies in
  Mathematical Physics}.
\newblock North-Holland Publishing Co., Amsterdam, 1991.
\newblock Finite-dimensional deterministic studies.

\bibitem{BPT}
C.~Buzzi, C.~Pessoa, and J.~Torregrosa.
\newblock Piecewise linear perturbations of a linear center, 2013.

\bibitem{MLNT}
J.~L. Cardoso, J.~Llibre, D.~D. Novaes, and D.~J. Tonon.
\newblock Simultaneous occurrence of sliding and crossing limit cycles in
  piecewise linear planar vector fields.
\newblock {\em Dynamical Systems}, 35(3):490--514, 2020.

\bibitem{CSN}
V.~Carmona, F.~Fernández-Sánchez, and D.~D. Novaes.
\newblock A new simple proof for lum–chua’s conjecture.
\newblock {\em Nonlinear Analysis: Hybrid Systems}, 40:100992, 2021.

\bibitem{VAR}
X.~Chen and K.-K. Tung.
\newblock Varying planetary heat sink led to global-warming slowdown and
  acceleration.
\newblock {\em Science}, 345(6199):897--903, 2014.

\bibitem{GCP1}
B.~Coll, A.~Gasull, and R.~Prohens.
\newblock Differential equations defined by the sum of two quasi-homogeneous
  vector fields.
\newblock {\em Canad. J. Math.}, 49(2):212--231, 1997.

\bibitem{GCP}
B.~Coll, A.~Gasull, and R.~Prohens.
\newblock Degenerate {H}opf bifurcations in discontinuous planar systems.
\newblock {\em J. Math. Anal. Appl.}, 253(2):671--690, 2001.

\bibitem{CT}
M.~C. Crucifix.
\newblock Oscillators and relaxation phenomena in pleistocene climate theory.
\newblock {\em Philosophical transactions. Series A, Mathematical, physical,
  and engineering sciences}, 370:1140 -- 1165, 2011.

\bibitem{Ocean}
H.~Dijkstra.
\newblock {\em Nonlinear Physical Oceanography: A Dynamical Systems Approach to
  the Large Scale Ocean Circulation and El Ni{\~n}o,}.
\newblock Springer Netherlands, 2009.

\bibitem{Fil1988}
A.~F. Filippov.
\newblock {\em Differential equations with discontinuous righthand sides},
  volume~18 of {\em Mathematics and its Applications (Soviet Series)}.
\newblock Kluwer Academic Publishers Group, Dordrecht, 1988.
\newblock Translated from the Russian.

\bibitem{Freire1998}
E.~Freire, E.~Ponce, F.~Rodrigo, and F.~Torres.
\newblock Bifurcation sets of continuous piecewise linear systems with two
  zones.
\newblock {\em Internat. J. Bifur. Chaos Appl. Sci. Engrg.}, 8(11):2073--2097,
  1998.

\bibitem{Ponce}
E.~Freire, E.~Ponce, and F.~Torres.
\newblock Canonical discontinuous planar piecewise linear systems.
\newblock {\em SIAM J. Appl. Dyn. Syst.}, 11(1):181--211, 2012.

\bibitem{FPTF}
E.~Freire, E.~Ponce, and F.~Torres.
\newblock The discontinuous matching of two planar linear foci can have three
  nested crossing limit cycles.
\newblock {\em Publicacions Matemàtiques}, EXTRA:221--253, 04 2014.

\bibitem{FPTF2}
E.~Freire, E.~Ponce, and F.~Torres.
\newblock A general mechanism to generate three limit cycles in planar filippov
  systems with two zones.
\newblock {\em Nonlinear Dynamics}, 78:251--263, 10 2014.

\bibitem{GatManRin1973}
M.~Gatto, D.~Mandrioli, and S.~Rinaldi.
\newblock Pseudoequilibrium in dynamical systems.
\newblock {\em Internat. J. Systems Sci.}, 4:809--824, 1973.

\bibitem{HanZhang}
M.~Han and W.~Zhang.
\newblock On hopf bifurcation in non-smooth planar systems.
\newblock {\em Journal of Differential Equations}, 248(9):2399--2416, 2010.

\bibitem{SongYang}
S.-M. Huan and X.-S. Yang.
\newblock On the number of limit cycles in general planar piecewise linear
  systems, 2012.

\bibitem{node}
S.-M. Huan and X.-S. Yang.
\newblock On the number of limit cycles in general planar piecewise linear
  systems of node–node types.
\newblock {\em Journal of Mathematical Analysis and Applications},
  411(1):340--353, 2014.

\bibitem{inter}
T.~Huck, A.~C. de~Verdière, and A.~J. Weaver.
\newblock Interdecadal variability of the thermohaline circulation in box-ocean
  models forced by fixed surface fluxes.
\newblock {\em Journal of Physical Oceanography}, 29(5):865 -- 892, 1999.

\bibitem{KapEng2013}
H.~Kaper and H.~Engler.
\newblock {\em Mathematics and climate}.
\newblock SIAM, 2013.

\bibitem{leifeld2016smooth}
J.~K. Leifeld.
\newblock {\em Smooth and Nonsmooth Bifurcations in Welander's Ocean Convection
  Model}.
\newblock PhD thesis, University of Minnesota, 2016.

\bibitem{LiP}
A.~M. Liapunov.
\newblock {\em Stability of motion}.
\newblock Mathematics in Science and Engineering, Vol. 30. Academic Press, New
  York-London, 1966.
\newblock With a contribution by V. A. Pliss and an introduction by V. P.
  Basov, Translated from the Russian by Flavian Abramovici and Michael
  Shimshoni.

\bibitem{LP}
J.~Llibre and E.~Ponce.
\newblock Three nested limit cycles in discontinuous piecewise linear
  differential systems with two zones.
\newblock {\em Dynamics of Continuous Discrete and Impulsive Systems-series
  B-applications \& Algorithms}, 19:0325--335, 2012.

\bibitem{Chua}
R.~Lum and L.~O. Chua.
\newblock Global properties of continuous piecewise linear vector fields. part
  i: Simplest case in $\mathbb{R}^2$.
\newblock {\em International Journal of Circuit Theory and Applications},
  19(3):251--307, 1991.

\bibitem{MakLam2012}
O.~Makarenkov and J.~S.~W. Lamb.
\newblock Dynamics and bifurcations of nonsmooth systems: a survey.
\newblock {\em Phys. D}, 241(22):1826--1844, 2012.

\bibitem{NLT}
D.~Novaes, J.~Libre, and M.~Teixeira.
\newblock Limit cycles bifurcating from the periodic orbits of a discontinuous
  piecewise linear differentiable center with two zones.
\newblock {\em International Journal of Bifurcation and Chaos}, 25:1550144 (11
  pages), 10 2015.

\bibitem{NT}
D.~D. Novaes and J.~Torregrosa.
\newblock On extended chebyshev systems with positive accuracy.
\newblock {\em Journal of Mathematical Analysis and Applications},
  448(1):171--186, 2017.

\bibitem{Dyn}
K.~Sakai and W.~R. Peltier.
\newblock A dynamical systems model of the dansgaard–oeschger oscillation and
  the origin of the bond cycle.
\newblock {\em Journal of Climate}, 12(8):2238 -- 2255, 1999.

\bibitem{SimJohn2010}
D.~J.~W. Simpson.
\newblock {\em Bifurcations in piecewise-smooth continuous systems}, volume~70
  of {\em World Scientific Series on Nonlinear Science. Series A: Monographs
  and Treatises}.
\newblock World Scientific Publishing Co. Pte. Ltd., Hackensack, NJ, 2010.

\bibitem{using}
K.-K. Tung and J.~Zhou.
\newblock Using data to attribute episodes of warming and cooling in
  instrumental records.
\newblock {\em Proceedings of the National Academy of Sciences},
  110(6):2058--2063, 2013.

\bibitem{welander}
P.~Welander.
\newblock A simple heat-salt oscillator.
\newblock {\em Dynamics of Atmospheres and Oceans}, 6(4):233--242, 1982.

\end{thebibliography}

\end{document}